\PassOptionsToPackage{square,comma,numbers,sort&compress}{natbib}
\documentclass{article}

\usepackage[final]{neurips_2021}
\usepackage{bbding}

\usepackage[utf8]{inputenc} % allow utf-8 input
\usepackage[T1]{fontenc}    % use 8-bit T1 fonts
\usepackage{url}            % simple URL typesetting
\usepackage{booktabs}       % professional-quality tables
\usepackage{amsfonts}       % blackboard math symbols
\usepackage{nicefrac}       % compact symbols for 1/2, etc.
\usepackage{microtype}      % microtypography
\usepackage{wrapfig}
\usepackage{amssymb}

\usepackage{amsmath}
\usepackage{amsfonts}
\usepackage{algorithm}
\usepackage{algorithmic}

\newtheorem{theorem}{Theorem}
\newtheorem{lemma}{Lemma}

\newtheorem{definition}{Definition}

\newtheorem{proof}{Proof}

\usepackage{times}
\usepackage{epsfig}
\usepackage{graphicx}
\usepackage{amsmath}
\usepackage{amssymb}

\usepackage{amsfonts}
\usepackage{algorithm}
\usepackage{algorithmic}
\usepackage{color}

\usepackage{lineno,hyperref}
\usepackage{amsfonts}
\usepackage{multirow}
\usepackage{bm}

\usepackage{diagbox}
\usepackage{multirow}

\usepackage{times}
\usepackage{soul}
\usepackage{url}
\usepackage{algorithm}
\usepackage{algorithmic}
\usepackage{subfigure}
\usepackage{lineno,hyperref}
\usepackage{multirow}
\usepackage{bm}
\modulolinenumbers[1]
\usepackage{rotating}
\usepackage{booktabs}
\usepackage{makecell}

\usepackage{wrapfig,lipsum,booktabs}

\usepackage{amsfonts}       % blackboard math symbols
\usepackage{nicefrac}       % compact symbols for 1/2, etc.
\usepackage{microtype}      % microtypography
\usepackage{xcolor}         % colors

\title{Sufficient D-Stability Conditions for Non-Square Matrices}
\author{Yuhao Tong$^1$, Steven W. Su $^{1,2}$ $^*$
\AND
\thanks{$^{*}$ The co-responding author.}
\thanks{$^{1}$ College of Artificial Intelligence and Big data for Medical Sciences, Shandong First Medical University
Shandong Academy of Medical Sciences, P. R. China.}
\thanks{$^{2}$ University of Technology, Sydney, Australia.}
}

\begin{document}

\maketitle

\begin{abstract}
This note explores the extension of D-stability to non-square matrices, applicable to distributed/decentralized controllability analysis. We first present a definition of D-stability for non-square matrices, directly extending from square matrices. We propose sufficient conditions for specific configurations of non-square matrices. Finally, we consider the selection of configurations to ensure the D-stability of a given non-square system. 
\end{abstract}

\section{INTRODUCTION}

In various applications, the stability of systems under different structural variations and constraints is a major concern in relevant research areas. For instance, in the linearization of diffusion models in biological systems at constant equilibrium, strongly stable matrices are proposed \cite{CROSS1978253, hadeler2006nonlinear}. The field of system control and its applications, including decentralized stability \cite{sun2023gallery, Wang:-On, hadeler2006nonlinear, johnson1974sufficient, goh1977global, iggidr2023limits, barker1978positive}, has motivated research into the stability of various matrices, such as P-matrices, D-stable matrices, and Volterra-Lyapunov stable matrices. These matrices are also significant in other areas, like economics.

In the process control industry, decentralized or distributed control is often preferred over centralized control due to its simplicity and fault tolerance. Consequently, the configuration of decentralized control systems has been extensively investigated. Topics such as Decentralized Unconditional Stability and Decentralized Controllability have been well-explored \cite{anderson1981algebraic, Anderson:-Alge, Sko:-ariab, grosdidier1986interaction, Su:-Analy}.

Particularly in applications involving non-square systems, addressing the stability of certain special real matrices associated with these systems is essential \cite{zhang2017multiloop, WHEATON20171, zhiteckii2022robust, sujatha2022control, steentjes2022data}. This necessity has led to the extension of stability concepts to non-square matrices. This note focuses exclusively on real matrices, with the potential for extending to complex matrices via frequency-based techniques.

This study primarily focuses on the conditions to ensure the extended D-Stability for non-square matrices. The major contribution is the proof of a mild sufficient D-Stability condition for non-square matrices. We also propose an open problem for investigating the necessary and sufficient conditions of D-Stability for non-square matrices.

\section{DECENTRALIZED STABILIZATION FOR NON-SQUARE SYSTEMS AND MATRICES}

One of the motivations for extending decentralized stability results from square to non-square matrices stems from exploring the decentralized unconditional stability (DUS) of non-square processes \cite{Sko:-ariab, grosdidier1986interaction, Su:-Analy}. A key DUS condition mirrors the D-stability requirement for the system's steady state gain matrix, typically a non-square real matrix in non-square processes. Here, we present a D-stability-like sufficient condition for non-square matrices to achieve decentralized unconditional stability. While this note primarily focuses on the investigation of real non-square matrices, we will not provide details on applying this condition to ensure DUS in non-square processes. Instead, our focus is on linking DUS conditions of the steady state transfer functions between non-square systems and their corresponding square systems. For readers interested in the proof of this sufficient DUS condition, we recommend consulting existing literature that discusses the proof of DUS conditions for square systems using singular perturbation analysis \cite{Kokotovic:-Sing}.

We first present the definition of D-stability for a non-square matrix as follows.

\begin{definition} \label{Dfn1}
	For a non-square matrix $A \in \mathbb{R}^{m \times n}$ with $n > m$, we define $A$ as D-stable if there exists a block diagonal non-square matrix $K \in \mathbb{R}^{n \times m}$ such that for all non-negative diagonal matrices $E \in  \mathbb{R}^{n \times n}$, and $j \in \mathcal{M}$ \cite{campo1994achievable},
	\begin{equation}\label{suf_con}
		\text{Re}\{ \sigma_i ( [A E K]_j ) \} >0,
	\end{equation}
	where $\mathcal{M}$ is the index set consisting of $k$ tuples of integers in the range $1,\cdots, m$, and $\text{Re} \{ \sigma_i (M) \}$ denotes the real part of the $i$-th eigenvalue of matrix $M$.
\end{definition}

Definition \ref{Dfn1} describes the characteristic of a non-square D-stable-like matrix. Based on this extension, we now investigate the D-stable conditions between non-square and square matrices. To formalize our discussion, we define the structure of the $K$ matrix. According to this structure, we will define the square matrices for the non-square matrices $A$ and $E$.

First, we assume the structure of the block diagonal $K$ matrix as follows:

\begin{equation}\label{integral_matrix}
K=
\begin{bmatrix}
    k_{1,1} & \cdots & k_{1, p_1} & 0 & \cdots & 0 & \cdots & 0 & \cdots & 0 \\
    0 & \cdots & 0 & k_{2, 1} & \cdots & k_{2, p_2} & \cdots & 0 & \cdots & 0\\
    \vdots & \vdots & \ddots & \vdots & \vdots & \ddots & \vdots & \vdots & \ddots & \vdots \\
    0 & \cdots & 0 & 0 & \dots & 0 & \cdots & k_{m, 1} & \cdots & k_{m, p_m}
\end{bmatrix}^T
\end{equation}
where $\sum_{i=1}^{m} p_i = n$. To simplify our discussion, we assume all the elements $k_{i,j}$ are non-negative.

Considering the non-negative diagonal matrix $ E= diag\{ [\varepsilon _{1,1}, \cdots,
  \varepsilon_{m,p_m}] \} \,\,\, \,\, $ with $\varepsilon _{i,j} \ge 0 $, we have
\begin{equation}\label{integral_matrix_E}
\begin{aligned}
\bar K& = E K=\\
&\begin{bmatrix}
    k_{1,1} \varepsilon _{1,1} & \cdots & k_{1, p_1} \varepsilon _{1, p_1} & 0 & \cdots & 0 & \cdots & 0 & \cdots & 0 \\
    0 & \cdots & 0 & k_{2, 1} \varepsilon _{2, 1} & \cdots & k_{2, p_2} \varepsilon _{2, p_2} & \cdots & 0 & \cdots & 0\\
    \vdots & \vdots & \ddots & \vdots & \vdots & \ddots & \vdots & \vdots & \ddots & \vdots \\
    0 & \cdots & 0 & 0 & \dots & 0 & \cdots & k_{m, 1} \varepsilon _{m,1} & \cdots & k_{m, p_m} \varepsilon _{m, p_m}
\end{bmatrix}^T
\end{aligned}
\end{equation}

According to the structure of $K$, we redefine the sub-indices of the $n$ columns of $A$ as follows:
$$A = [\boldsymbol{a}_{1,1}, \cdots, \boldsymbol{a}_{1,p_1}, \boldsymbol{a}_{2,1}, \cdots, \boldsymbol{a}_{2,p_2}, \cdots, \cdots, \boldsymbol{a}_{m,1}, \cdots, \boldsymbol{a}_{m,p_m}],$$
then $AEK$ as shown in (\ref{suf_con}) can be expressed as
\begin{equation}\label{eq_li_comb}
	AEK = \left[ \sum_{i=1}^{p_1} \varepsilon_{1,i} k_{1,i} \boldsymbol{a}_{1,i}, \sum_{i=1}^{p_2} \varepsilon_{2, i} k_{2,i} \boldsymbol{a}_{2,i}, \cdots, \sum_{i=1}^{p_m} \varepsilon_{m,i} k_{m,i} \boldsymbol{a}_{m,i} \right].
\end{equation}

Now, based on the structure in (\ref{integral_matrix}), we define the Squared Matrices for the matrices $A$ and $K$ as follows.

%\begin{definition} (Squared Matrices) \label{squaredmatrix}
%For a non-square matrix $A \in \mathbb{R}^{m \times n}$ and its associated block diagonal matrix $K \in \mathbb{R}^{m \times n}$ as defined in Equation (\ref{integral_matrix}),  we follow the following procedure to construct squared matrices based on the values of $\varepsilon_{i,j}$.

%First, we consider the case where the dimension of the squared matrix is $m$, i.e., $\forall i \in \{1, 2, \cdots, m \}$, we select only one of the column from $\boldsymbol{a}_{i, j}$, $j \in \{1, 2, \cdots, p_i\}$. This is corresponding to the case that for a specific $i$, only one of the corresponding $\varepsilon_{i,j} \ne 0$,  $j \in \{1, 2, \cdots, p_i\}$. Thus, we can construct $N=\Pi_{1}^{m} p_i$ squared matrices for both matrices $A$ and $K$, we note this squared matrices as $[A]^m_{s_i}$ and $[K]^m_{s_i}$ $i \in \{1, 2, \cdots， N\}$.

%Second, when the dimension of the squared matrices is less than $m$, which corresponds to the case that for one or more $i \in \{1, 2, \cdots, m \}$, $\varepsilon_{i,j} = 0$,  $\forall j \in \{1, 2, \cdots, p_i\}$. Then, we can construct corresponding squared matrices with dimension $k < m$ as $[A]^k_{s_i}$ and $[K]^k_{s_i}$.

%We refer to all these square matrices as the squared matrices corresponding to their non-square counterparts.
%\end{definition}

\begin{definition} (Squared Matrices) \label{squaredmatrix}
Consider a non-square matrix $A \in \mathbb{R}^{m \times n}$ ($n > m$) and its associated block diagonal matrix $K \in \mathbb{R}^{n \times m}$, as defined in Equation (\ref{integral_matrix}). We describe a procedure to construct squared matrices from $A$ and $K$ regarding the values of $\varepsilon_{i,j}$.

First, for the case where the dimension of the squared matrix is $m$, i.e., $\forall i \in \{1, 2, \cdots, m\}$, we select only one column from each $\boldsymbol{a}_{i, j}$, where $j \in \{1, 2, \cdots, p_i\}$. This corresponds to the scenario where, for a given $i$, exactly one of the corresponding $\varepsilon_{i,j} \ne 0$, $j \in \{1, 2, \cdots, p_i\}$. We can construct $N = \prod_{i=1}^{m} p_i$ such squared matrices for both $A$ and $K$, denoted as $[A]^m_{s_l}$ and $[K]^m_{s_l}$, respectively, for $l \in \{1, 2, \cdots, N\}$.

Second, when the dimension of the squared matrices is less than $m$, corresponding to the scenario where, for one or more $i \in \{1, 2, \cdots, m\}$, $\varepsilon_{i,j} = 0$ for all $j \in \{1, 2, \cdots, p_i\}$, we can construct squared matrices of dimension $k < m$. These are denoted as $[A]^k_{s_l}$ and $[K]^k_{s_l}$.

These squared matrices are referred to as the square counterparts of the non-square matrices $A$ and $K$.
\end{definition}

In Definition \ref{squaredmatrix},  regarding the original decentralized stabilization problem, the reduction in the dimension of matrices occurs in cases where all the control inputs for certain columns have zero gain. In other words, for some index $i$, if $\varepsilon_{i,j} = 0$ for all $j \in {1, 2, \cdots, p_i}$, it corresponds to the removal of a control input from service, as discussed in \cite{campo1994achievable}.

Based on the definition of the squared matrices (Definition \ref{squaredmatrix}), we present the connection between square and non-square matrices.

First, we introduce two well-known stable square matrices.

\begin{definition} \cite{Cross:-Thre}
The matrix $A=(a_{ij}) \in \mathbb{R}^{n \times n}$ is said to be

(1) D-stable if $A D$ is stable for all diagonal matrix $D>0$;

(2) Volterra-Lyapunov stable if there exists a diagonal matrix $D >0$ for which $AD + DA^T > 0$.
\end{definition}

\begin{lemma} (\cite{Cross:-Thre}) \label{lemma2}
If a matrix $A=(a_{ij}) \in \mathbb{R}^{n \times n}$ is Volterra-Lyapunov stable, then, for any diagonal positive matrix $D > 0$ and $i \in \mathcal{M}$, where $\mathcal{M}$ is the index set consisting of $k$ tuples of integers in the
range $1,\cdots, m$,
$$Re\{ \sigma [AD]_i \} >0,$$
which implies that $[A]_i$, for $i \in \mathcal{M}$,  are all D-stable matrices.
\end{lemma}

\begin{definition} \label{def_NSQ_lyp}
	For a real nonsquare matrix $A \in \mathbb{R}^{m \times n}$, if one can find a diagonal positive matrix  $D > 0$, so that all its $m$-order squared matrices (see Definition \ref{squaredmatrix}) $[A]^m_{s_i} \in \mathbb{R}^{m \times m}$  satisfy the following LMIs,
	$$[A]^m_{s_i} D + D ([A]^m_{s_i})^T  > 0, $$
	then, we call the non-square matrix $A$ a simultaneous Volterra-Lyapunov stable matrix.
\end{definition}

\begin{definition} \label{def_NSQ_lyp_ind}
For a real nonsquare matrix $A \in \mathbb{R}^{m \times n}$, if one can find an individual diagonal positive matrix  $D_i > 0$, so that all its $m$-order squared matrices (see Definition \ref{squaredmatrix}) $[A]^m_{s_i} \in \mathbb{R}^{m \times m}$  satisfy the following LMIs,
$$[A]^m_{s_i} D_i + D_i ([A]^m_{s_i})^T  > 0, $$
then, we call the non-square matrix $A$ an individual Volterra-Lyapunov stable matrix.
\end{definition}

%%%%%%%%%%%%%%%%%%%%%%%%%%%%%%%%%%%%%%%%%%%%%%%%

\begin{lemma} \label{simul_lemma22a}
	If a real nonsquare matrix $A \in \mathbb{R}^{m \times n}$ is an individual Volterra-Lyapunov stable matrix (see Definition \ref{def_NSQ_lyp_ind}),
	then, there exists a
	block diagonal non-square matrix $K \in \mathbb{C}^{n \times m}$ such that for all non-negative diagonal matrices $E \in  \mathbb{R}^{n \times n}$, and $j \in \mathcal{M}$, where $\mathcal{M}$ is the index set consisting of $k$ tuples of integers in the
	range $1,\cdots, m$,
	\begin{equation}\label{suf_con_1}
		Re\{ \sigma_i ( [A E K]_j ) \} >0,
	\end{equation}
	where $\text{Re} \{ \sigma_i (M) \}$ represents the real part of the $i$-th eigenvalue of matrix $M$.
\end{lemma}

%Assume $A=[h_1, h_2, \cdots h_n]$, then we can see
%\begin{equation}\label{eq_li_comb}
%A E \bar K = [ \sum_{i=1}^{p1} \epsilon_i k_{1 i} h_i, \sum_{i=p_1+1}^{p_2} \epsilon_i k_{2 i} h_i, \cdots, \sum_{i=p_{m-1}+1}^{n} \epsilon_i k_{m i} h_i].
%\end{equation}
%In the inequality (\ref{eq_important}), it can be seen that if we well select $\gamma_i$ ($i\in \{1,\cdots, N\}$) and $D$, we can always ensure that
%\begin{equation}\label{eq_sl}
%\begin{aligned}
%&(\sum_{i=1}^{N} \gamma_i  [A]^m_{s_i})  D = A E \bar K = \\
%&[ \sum_{i=1}^{p1} \epsilon_i k_{1 i} h_i, \sum_{i=p_1+1}^{p_2} \epsilon_i k_{2 i} h_i, \cdots, \sum_{i=p_{m-1}+1}^{n} \epsilon_i k_{m i} h_i].
%\end{aligned}
%\end{equation}
%
%Based on both (\ref{eq_important}) and (\ref{eq_sl}), we get (\ref{suf_con}).
%%Without loss of generality, we assume  Equation (\ref{eq_li_comb}),

%Without loss of generality, we assume  Equation (\ref{eq_li_comb}),

To prove Lemma \ref{simul_lemma22a}, we first introduce Lemma \ref{pailiezuhe_Lemma}, which will aid in the proof. For clarity and generality, we employ the scenario of gambling as an illustrative example. This example will later be shown to be directly relevant to the proof of Lemma \ref{simul_lemma22a}.

%%%%%%%%%%%%%%%%%%%%%%%%%%%%%%%%%%%%%%%%%%%%%%%%%%%%%%
\begin{lemma} \label{pailiezuhe_Lemma}
Consider $m$ groups of cards. Let $\phi$ denote the current group index with $0 < \phi \leq m$, where each group contains $p_\phi$ distinct cards. A player can place a bet on each combination of \(m\) cards, selecting one card from each group. Let \(\kappa_{\phi}\) be the card chosen from the \(\phi\)-th group. The payoff for a specific combination is denoted by $\gamma_{\kappa_1, \kappa_2, \ldots, \kappa_m} \ge 0$. The payoff for each card in any combination can be calculated by the product of \(\gamma\) and a positive proportionality factor \(\lambda\), which varies both with the combination and within a combination among different cards. Consequently, the payoff for each card differs across combinations and also varies within a single combination due to differences in \(\lambda\). The total payoff for each card is the sum of its payoffs from all combinations. Suppose the player desires a specific ratio \(k_j^\phi > 0\) between the payoff of the \((j+1)\)-th card of group \(\phi\) and the first card of that group.

Then, it is always possible to find a set of positive values for $\gamma_{\kappa_1, \kappa_2, \ldots, \kappa_m} \ge 0$ such that the desired ratios $k_j^\phi$ are achieved for each set.

\end{lemma}

\begin{lemma} \label{pailiezuhe_Lemma}
	Consider $m$ groups of cards. Let $\phi$ denote the current group index with $0 < \phi \leq m$, where each group contains $p_\phi$ distinct cards. A player can place a bet on each combination of \(m\) cards, selecting one card from each group. Let \(\kappa_{\phi}\) be the card chosen from the \(\phi\)-th group. The payoff for a specific combination is denoted by $\gamma_{\kappa_1, \kappa_2, \ldots, \kappa_m} \ge 0$. The payoff for each card in any combination can be calculated by the product of $\gamma_{\kappa_1, \kappa_2, \ldots, \kappa_m}$ and a positive proportionality factor \(\lambda\), which varies both with the combination and within a combination among different cards. Consequently, the payoff for each card differs across combinations and also varies within a single combination due to differences in \(\lambda\). The total payoff for each card is the sum of its payoffs from all combinations. Suppose the player desires a specific ratio \(k_j^\phi > 0\) between the payoff of the \((j+1)\)-th card of group \(\phi\) and the first card of that group.
	
	Then, it is always possible to find a set of positive values for $\gamma_{\kappa_1, \kappa_2, \ldots, \kappa_m} \ge 0$ such that the desired ratios $k_j^\phi$ are achieved for each set.
\end{lemma}

\begin{proof}
See Appendix A.
\end{proof}

The proof of Lemma \ref{pailiezuhe_Lemma} is constructive, involving the creation of all necessary $\gamma$ values and verifying they meet the lemma's requirements. To further substantiate the theorem, we present a numerical example that highlights the effectiveness of the proposed constructive method.

\subsection{Illustrative Example}

Consider an example with $3$ groups, each containing $3$ elements. We need to traverse and combine each element in each group without repetition, resulting in a total of $27 $combinations, denoted as \( \gamma_{1,1,1} \) to \( \gamma_{3,3,3} \). There are $81$ parameters \( \lambda \). For the proportions \( k \), there are a total of $6$, denoted as \( k^{1}_{1} \), \( k^{1}_{2} \), \( k^{2}_{1} \), \( k^{2}_{2} \), \( k^{3}_{1} \), \( k^{3}_{2} \).

Using this setup, we can explore how to express and simplify the relationships between these \( k \) ratios using the \( \gamma \) and \( \lambda \) parameters.

For the proportions \(k^{1}_{1}\) and \(k^{1}_{2}\), two equations can be established.

\scalebox{0.92}{%
$\frac{\gamma_{1,1,1}\lambda_{1,1,1}^{1}+\gamma_{1,1,2}\lambda_{1,1,2}^{1}+\gamma_{1,1,3}\lambda_{1,1,3}^{1}+\gamma_{1,2,1}\lambda_{1,2,1}^{1}+\gamma_{1,2,2}\lambda_{1,2,2}^{1}+\gamma_{1,2,3}\lambda_{1,2,3}^{1}+\gamma_{1,3,1}\lambda_{1,3,1}^{1}+\gamma_{1,3,2}\lambda_{1,3,2}^{1}+\gamma_{1,3,3}\lambda_{1,3,3}^{1}}{\gamma_{2,1,1}\lambda_{2,1,1}^{1}+\gamma_{2,1,2}\lambda_{2,1,2}^{1}+\gamma_{2,1,3}\lambda_{2,1,3}^{1}+\gamma_{2,2,1}\lambda_{2,2,1}^{1}+\gamma_{2,2,2}\lambda_{2,2,2}^{1}+\gamma_{2,2,3}\lambda_{2,2,3}^{1}+\gamma_{2,3,1}\lambda_{2,3,1}^{1}+\gamma_{2,3,2}\lambda_{2,3,2}^{1}+\gamma_{2,3,3}\lambda_{2,3,3}^{1}}=k^{1}_{1}
$}

\scalebox{0.92}{%
$\frac{\gamma_{1,1,1}\lambda_{1,1,1}^{1}+\gamma_{1,1,2}\lambda_{1,1,2}^{1}+\gamma_{1,1,3}\lambda_{1,1,3}^{1}+\gamma_{1,2,1}\lambda_{1,2,1}^{1}+\gamma_{1,2,2}\lambda_{1,2,2}^{1}+\gamma_{1,2,3}\lambda_{1,2,3}^{1}+\gamma_{1,3,1}\lambda_{1,3,1}^{1}+\gamma_{1,3,2}\lambda_{1,3,2}^{1}+\gamma_{1,3,3}\lambda_{1,3,3}^{1}}{\gamma_{3,1,1}\lambda_{3,1,1}^{1}+\gamma_{3,1,2}\lambda_{3,1,2}^{1}+\gamma_{3,1,3}\lambda_{3,1,3}^{1}+\gamma_{3,2,1}\lambda_{3,2,1}^{1}+\gamma_{3,2,2}\lambda_{3,2,2}^{1}+\gamma_{3,2,3}\lambda_{3,2,3}^{1}+\gamma_{3,3,1}\lambda_{3,3,1}^{1}+\gamma_{3,3,2}\lambda_{3,3,2}^{1}+\gamma_{3,3,3}\lambda_{3,3,3}^{1}}=k^{1}_{2}
$}

For the proportions \(k^{1}_{1}\) and \(k^{1}_{2}\), we arrange the product of \(\gamma\) and \(\lambda\) in the numerator and denominator in ascending order of \(\gamma\) subscripts, ensuring each product term corresponds one-to-one. The following equations are established:

For \( k^{1}_{1} \):

\[
\begin{aligned}
	&\frac{\gamma_{1,1,1} \lambda_{1,1,1}^{1}}{\gamma_{2,1,1} \lambda_{2,1,1}^{1}} = k^{1}_{1}, \quad \frac{\gamma_{1,1,2} \lambda_{1,1,2}^{1}}{\gamma_{2,1,2} \lambda_{2,1,2}^{1}} = k^{1}_{1}, \quad \frac{\gamma_{1,1,3} \lambda_{1,1,3}^{1}}{\gamma_{2,1,3} \lambda_{2,1,3}^{1}} = k^{1}_{1}, \\
	&\frac{\gamma_{1,2,1} \lambda_{1,2,1}^{1}}{\gamma_{2,2,1} \lambda_{2,2,1}^{1}} = k^{1}_{1}, \quad \frac{\gamma_{1,2,2} \lambda_{1,2,2}^{1}}{\gamma_{2,2,2} \lambda_{2,2,2}^{1}} = k^{1}_{1}, \quad \frac{\gamma_{1,2,3} \lambda_{1,2,3}}{\gamma_{2,2,3} \lambda_{2,2,3}} = k^{1}_{1}, \\
	&\frac{\gamma_{1,3,1} \lambda_{1,3,1}^{1}}{\gamma_{2,3,1} \lambda_{2,3,1}^{1}} = k^{1}_{1}, \quad \frac{\gamma_{1,3,2} \lambda_{1,3,2}^{1}}{\gamma_{2,3,2} \lambda_{2,3,2}^{1}} = k^{1}_{1}, \quad \frac{\gamma_{1,3,3} \lambda_{1,3,3}^{1}}{\gamma_{2,3,3} \lambda_{2,3,3}^{1}} = k^{1}_{1}.
\end{aligned}
\]

Similarly, for \( k^{1}_{2} \):

\[
\begin{aligned}
	&\frac{\gamma_{1,1,1} \lambda_{1,1,1}^{1}}{\gamma_{3,1,1} \lambda_{3,1,1}^{1}} = k^{1}_{2}, \quad \frac{\gamma_{1,1,2} \lambda_{1,1,2}^{1}}{\gamma_{3,1,2} \lambda_{3,1,2}^{1}} = k^{1}_{2}, \quad \frac{\gamma_{1,1,3} \lambda_{1,1,3}^{1}}{\gamma_{3,1,3} \lambda_{3,1,3}^{1}} = k^{1}_{2}, \\
	&\frac{\gamma_{1,2,1} \lambda_{1,2,1}^{1}}{\gamma_{3,2,1} \lambda_{3,2,1}^{1}} = k^{1}_{2}, \quad \frac{\gamma_{1,2,2} \lambda_{1,2,2}^{1}}{\gamma_{3,2,2} \lambda_{3,2,2}^{1}} = k^{1}_{2}, \quad \frac{\gamma_{1,2,3} \lambda_{1,2,3}^{1}}{\gamma_{3,2,3} \lambda_{3,2,3}^{1}} = k^{1}_{2}, \\
	&\frac{\gamma_{1,3,1} \lambda_{1,3,1}^{1}}{\gamma_{3,3,1} \lambda_{3,3,1}^{1}} = k^{1}_{2}, \quad \frac{\gamma_{1,3,2} \lambda_{1,3,2}^{1}}{\gamma_{3,3,2} \lambda_{3,3,2}^{1}} = k^{1}_{2}, \quad \frac{\gamma_{1,3,3} \lambda_{1,3,3}^{1}}{\gamma_{3,3,3} \lambda_{3,3,3}^{1}} = k^{1}_{2}.
\end{aligned}
\]

Using \(\gamma_{1,1,1}\), \(\gamma_{1,1,2}\), \(\gamma_{1,1,3}\), \(\gamma_{1,2,1}\), \(\gamma_{1,2,2}\), \(\gamma_{1,2,3}\), \(\gamma_{1,3,1}\), \(\gamma_{1,3,2}\), and \(\gamma_{1,3,3}\), we can represent the remaining \(18 \gamma\) terms.

For the proportions \(k^{2}_{1}\) and \(k^{2}_{2}\), two equations can be established.

\scalebox{0.9}{%
$\frac{\gamma_{1,1,1}\lambda_{1,1,1}^{2}+\gamma_{1,1,2}\lambda_{1,1,2}^{2}+\gamma_{1,1,3}\lambda_{1,1,3}^{2}+\gamma_{2,1,1}\lambda_{2,1,1}^{2}+\gamma_{2,1,2}\lambda_{2,1,2}^{2}+\gamma_{2,1,3}\lambda_{2,1,3}^{2}+\gamma_{3,1,1}\lambda_{3,1,1}^{2}+\gamma_{3,1,2}\lambda_{3,1,2}^{2}+\gamma_{3,1,3}\lambda_{3,1,3}^{2}}{\gamma_{1,2,1}\lambda_{1,2,1}^{2}+\gamma_{1,2,2}\lambda_{1,2,2}^{2}+\gamma_{1,2,3}\lambda_{1,2,3}^{2}+\gamma_{2,2,1}\lambda_{2,2,1}^{2}+\gamma_{2,2,2}\lambda_{2,2,2}^{2}+\gamma_{2,2,3}\lambda_{2,2,3}^{2}+\gamma_{3,2,1}\lambda_{3,2,1}^{2}+\gamma_{3,2,2}\lambda_{3,2,2}^{2}+\gamma_{3,2,3}\lambda_{3,2,3}^{2}}=k^{2}_{1}
$}

\scalebox{0.9}{%
$\frac{\gamma_{1,1,1}\lambda_{1,1,1}^{2}+\gamma_{1,1,2}\lambda_{1,1,2}^{2}+\gamma_{1,1,3}\lambda_{1,1,3}^{2}+\gamma_{2,1,1}\lambda_{2,1,1}^{2}+\gamma_{2,1,2}\lambda_{2,1,2}^{2}+\gamma_{2,1,3}\lambda_{2,1,3}^{2}+\gamma_{3,1,1}\lambda_{3,1,1}^{2}+\gamma_{3,1,2}\lambda_{3,1,2}^{2}+\gamma_{3,1,3}\lambda_{3,1,3}^{2}}{\gamma_{1,3,1}\lambda_{1,3,1}^{2}+\gamma_{1,3,2}\lambda_{1,3,2}^{2}+\gamma_{1,3,3}\lambda_{1,3,3}^{2}+\gamma_{2,3,1}\lambda_{2,3,1}^{2}+\gamma_{2,3,2}\lambda_{2,3,2}^{2}+\gamma_{2,3,3}\lambda_{2,3,3}^{2}+\gamma_{3,3,1}\lambda_{3,3,1}^{2}+\gamma_{3,3,2}\lambda_{3,3,2}^{2}+\gamma_{3,3,3}\lambda_{3,3,3}^{2}}=k^{2}_{2}
$}

Given the formulas for \( k^{1}_{1} \) and \( k^{1}_{2} \), we have the substitutions:
\[
\begin{aligned}
	\gamma_{2,1,1} &= \frac{\gamma_{1,1,1} \lambda_{1,1,1}^{1}}{k^{1}_{1} \lambda_{2,1,1}^{1}}, \quad
	\gamma_{2,1,2} &= \frac{\gamma_{1,1,2} \lambda_{1,1,2}^{1}}{k^{1}_{1} \lambda_{2,1,2}^{1}}, \quad
	\gamma_{2,1,3} &= \frac{\gamma_{1,1,3} \lambda_{1,1,3}^{1}}{k^{1}_{1} \lambda_{2,1,3}^{1}}, \\
	\gamma_{2,2,1} &= \frac{\gamma_{1,2,1} \lambda_{1,2,1}^{1}}{k^{1}_{1} \lambda_{2,2,1}^{1}}, \quad
	\gamma_{2,2,2} &= \frac{\gamma_{1,2,2} \lambda_{1,2,2}^{1}}{k^{1}_{1} \lambda_{2,2,2}^{1}}, \quad
	\gamma_{2,2,3} &= \frac{\gamma_{1,2,3} \lambda_{1,2,3}^{1}}{k^{1}_{1} \lambda_{2,2,3}^{1}}, \\
	\gamma_{2,3,1} &= \frac{\gamma_{1,3,1} \lambda_{1,3,1}^{1}}{k^{1}_{1} \lambda_{2,3,1}^{1}}, \quad
	\gamma_{2,3,2} &= \frac{\gamma_{1,3,2} \lambda_{1,3,2}^{1}}{k^{1}_{1} \lambda_{2,3,2}^{1}}, \quad
	\gamma_{2,3,3} &= \frac{\gamma_{1,3,3} \lambda_{1,3,3}^{1}}{k^{1}_{1} \lambda_{2,3,3}^{1}}.
\end{aligned}
\]

\[
\begin{aligned}
	\gamma_{3,1,1} &= \frac{\gamma_{1,1,1} \lambda_{1,1,1}^{1}}{k^{1}_{2} \lambda_{3,1,1}^{1}}, \quad
	\gamma_{3,1,2} &= \frac{\gamma_{1,1,2} \lambda_{1,1,2}^{1}}{k^{1}_{2} \lambda_{3,1,2}^{1}}, \quad
	\gamma_{3,1,3} &= \frac{\gamma_{1,1,3} \lambda_{1,1,3}^{1}}{k^{1}_{2} \lambda_{3,1,3}^{1}}, \\
	\gamma_{3,2,1} &= \frac{\gamma_{1,2,1} \lambda_{1,2,1}^{1}}{k^{1}_{2} \lambda_{3,2,1}^{1}}, \quad
	\gamma_{3,2,2} &= \frac{\gamma_{1,2,2} \lambda_{1,2,2}^{1}}{k^{1}_{2} \lambda_{3,2,2}^{1}}, \quad
	\gamma_{3,2,3} &= \frac{\gamma_{1,2,3} \lambda_{1,2,3}^{1}}{k^{1}_{2} \lambda_{3,2,3}^{1}}, \\
	\gamma_{3,3,1} &= \frac{\gamma_{1,3,1} \lambda_{1,3,1}^{1}}{k^{1}_{2} \lambda_{3,3,1}^{1}}, \quad
	\gamma_{3,3,2} &= \frac{\gamma_{1,3,2} \lambda_{1,3,2}^{1}}{k^{1}_{2} \lambda_{3,3,2}^{1}}, \quad
	\gamma_{3,3,3} &= \frac{\gamma_{1,3,3} \lambda_{1,3,3}^{1}}{k^{1}_{2} \lambda_{3,3,3}^{1}}.
\end{aligned}
\]

All the required \(\gamma\) terms can be represented using \(\gamma_{1,1,1}\), \(\gamma_{1,1,2}\), \(\gamma_{1,1,3}\), \(\gamma_{1,2,1}\), \(\gamma_{1,2,2}\), \(\gamma_{1,2,3}\), \(\gamma_{1,3,1}\), \(\gamma_{1,3,2}\), and \(\gamma_{1,3,3}\) with positive coefficients. Therefore, the formulas for \(k^{1}_{1}\) and \(k^{1}_{2}\) can be expressed using these \(\gamma\) terms.

\scalebox{0.86}{%
$\frac{(\lambda_{1,1,1}^{2}+\frac{\lambda_{2,1,1}^{2}\lambda_{1,1,1}^{1}}{ \lambda_{2,1,1}^{1} k^{1}_{1}}+\frac{\lambda_{3,1,1}^{2}\lambda_{1,1,1}^{1}}{\lambda_{3,1,1}^{1}k^{1}_{2}})\gamma_{1,1,1}+(\lambda_{1,1,2}^{2}+\frac{\lambda_{2,1,2}^{2}\lambda_{1,1,2}^{1}}{\lambda_{2,1,2}^{1} k^{1}_{1}}+\frac{\lambda_{3,1,2}^{2}\lambda_{1,1,2}^{1}}{\lambda_{3,1,2}^{1}k^{1}_{2}})\gamma_{1,1,2}+(\lambda_{1,1,3}^{2}+\frac{\lambda_{2,1,3}^{2}\lambda_{1,1,3}^{1}}{\lambda_{2,1,3}^{1} k^{1}_{1}}+\frac{\lambda_{3,1,3}^{2}\lambda_{1,1,3}^{1}}{\lambda_{3,1,3}^{1}k^{1}_{2}})\gamma_{1,1,3}}{(\lambda_{1,2,1}^{2}+\frac{\lambda_{2,1,1}^{2}\lambda_{1,2,1}^{1}}{\lambda_{2,2,1}^{1}k^{1}_{1}}+\frac{\lambda_{3,2,1}^{2}\lambda_{1,2,1}^{1}}{\lambda_{3,2,1}^{1}k^{1}_{2}})\gamma_{1,2,1}+(\lambda_{1,2,2}^{2}+\frac{\lambda_{2,2,2}^{2}\lambda_{1,2,2}^{1}}{\lambda_{2,2,2}^{1}k^{1}_{1}}+\frac{\lambda_{3,2,2}^{2}\lambda_{1,2,2}^{1}}{\lambda_{3,2,2}^{1}k^{1}_{2}})\gamma_{1,2,2}+(\lambda_{1,2,3}^{2}+\frac{\lambda_{2,2,3}^{2}\lambda_{1,2,3}^{1}}{\lambda_{2,2,3}^{1}k^{1}_{1}}+\frac{\lambda_{3,2,3}^{2}\lambda_{1,2,3}^{1}}{\lambda_{3,2,3}^{1}k^{1}_{2}})\gamma_{1,2,3}}=k^{2}_{1}
$}

\scalebox{0.86}{%
$\frac{(\lambda_{1,1,1}^{2}+\frac{\lambda_{2,1,1}^{2}\lambda_{1,1,1}^{1}}{ \lambda_{2,1,1}^{1} k^{1}_{1}}+\frac{\lambda_{3,1,1}^{2}\lambda_{1,1,1}^{1}}{\lambda_{3,1,1}^{1}k^{1}_{2}})\gamma_{1,1,1}+(\lambda_{1,1,2}^{2}+\frac{\lambda_{2,1,2}^{2}\lambda_{1,1,2}^{1}}{\lambda_{2,1,2}^{1} k^{1}_{1}}+\frac{\lambda_{3,1,2}^{2}\lambda_{1,1,2}^{1}}{\lambda_{3,1,2}^{1}k^{1}_{2}})\gamma_{1,1,2}+(\lambda_{1,1,3}^{2}+\frac{\lambda_{2,1,3}^{2}\lambda_{1,1,3}^{1}}{\lambda_{2,1,3}^{1} k^{1}_{1}}+\frac{\lambda_{3,1,3}^{2}\lambda_{1,1,3}^{1}}{\lambda_{3,1,3}^{1}k^{1}_{2}})\gamma_{1,1,3}}{(\lambda_{1,3,1}^{2}+\frac{\lambda_{2,3,1}^{2}\lambda_{1,3,1}^{1}}{\lambda_{2,3,1}^{1}k^{1}_{1}}+\frac{\lambda_{3,3,1}^{2}\lambda_{1,3,1}^{1}}{\lambda_{3,3,1}^{1}k^{1}_{2}})\gamma_{1,3,1}+(\lambda_{1,3,2}^{2}+\frac{\lambda_{2,3,2}^{2}\lambda_{1,3,2}^{1}}{\lambda_{2,3,2}^{1}k^{1}_{1}}+\frac{\lambda_{3,3,2}^{2}\lambda_{1,3,2}^{1}}{\lambda_{3,3,2}^{1}k^{1}_{2}})\gamma_{1,3,2}+(\lambda_{1,3,3}^{2}+\frac{\lambda_{2,3,3}^{2}\lambda_{1,3,3}^{1}}{\lambda_{2,3,3}^{1}k^{1}_{1}}+\frac{\lambda_{3,3,3}^{2}\lambda_{1,3,3}^{1}}{\lambda_{3,3,3}^{1}k^{1}_{2}})\gamma_{1,3,3}}=k^{2}_{2}
$}

Following the same method as before, we arrange the products of \(\gamma\) and \(\lambda\) in the numerator and denominator in ascending order of \(\gamma\) subscripts, ensuring each product term corresponds one-to-one. We then set the ratio of corresponding product terms in the numerator and denominator equal to the current proportion \(k\). Using this approach, we can represent \(\gamma_{1,2,1}\) and \(\gamma_{1,3,1}\) with \(\gamma_{1,1,1}\), \(\gamma_{1,2,2}\) and \(\gamma_{1,3,2}\) with \(\gamma_{1,1,2}\), and \(\gamma_{1,2,3}\) and \(\gamma_{1,3,3}\) with \(\gamma_{1,1,3}\), all with positive coefficients.

For the proportions \(k^{3}_{1}\) and \(k^{3}_{2}\), two equations can be established .

\scalebox{0.9}{%
$\frac{\gamma_{1,1,1}\lambda_{1,1,1}^{3}+\gamma_{1,2,1}\lambda_{1,2,1}^{3}+\gamma_{1,3,1}\lambda_{1,3,1}^{3}+\gamma_{2,1,1}\lambda_{2,1,1}^{3}+\gamma_{2,2,1}\lambda_{2,2,1}^{3}+\gamma_{2,3,1}\lambda_{2,3,1}^{3}+\gamma_{3,1,1}\lambda_{3,1,1}^{3}+\gamma_{3,2,1}\lambda_{3,2,1}^{3}+\gamma_{3,3,1}\lambda_{3,3,1}^{3}}{\gamma_{1,1,2}\lambda_{1,1,2}^{3}+\gamma_{1,2,2}\lambda_{1,2,2}^{3}+\gamma_{1,3,2}\lambda_{1,3,2}^{3}+\gamma_{2,1,2}\lambda_{2,1,2}^{3}+\gamma_{2,2,2}\lambda_{2,2,2}^{3}+\gamma_{2,3,2}\lambda_{2,3,2}^{3}+\gamma_{3,1,2}\lambda_{3,1,2}^{3}+\gamma_{3,2,2}\lambda_{3,2,2}^{3}+\gamma_{3,3,2}\lambda_{3,3,2}^{3}}=k^{3}_{1}
$}

\scalebox{0.9}{%
$\frac{\gamma_{1,1,1}\lambda_{1,1,1}^{3}+\gamma_{1,2,1}\lambda_{1,2,1}^{3}+\gamma_{1,3,1}\lambda_{1,3,1}^{3}+\gamma_{2,1,1}\lambda_{2,1,1}^{3}+\gamma_{2,2,1}\lambda_{2,2,1}^{3}+\gamma_{2,3,1}\lambda_{2,3,1}^{3}+\gamma_{3,1,1}\lambda_{3,1,1}^{3}+\gamma_{3,2,1}\lambda_{3,2,1}^{3}+\gamma_{3,3,1}\lambda_{3,3,1}^{3}}{\gamma_{1,1,3}\lambda_{1,1,3}^{3}+\gamma_{1,2,3}\lambda_{1,2,3}^{3}+\gamma_{1,3,3}\lambda_{1,3,3}^{3}+\gamma_{2,1,3}\lambda_{2,1,3}^{3}+\gamma_{2,2,3}\lambda_{2,2,3}^{3}+\gamma_{2,3,3}\lambda_{2,3,3}^{3}+\gamma_{3,1,3}\lambda_{3,1,3}^{3}+\gamma_{3,2,3}\lambda_{3,2,3}^{3}+\gamma_{3,3,3}\lambda_{3,3,3}^{3}}=k^{3}_{2}
$}

Since we can represent the remaining  \(18\gamma\) terms using \(\gamma_{1,1,1}\), \(\gamma_{1,1,2}\), \(\gamma_{1,1,3}\), \(\gamma_{1,2,1}\), \(\gamma_{1,2,2}\), \(\gamma_{1,2,3}\), \(\gamma_{1,3,1}\), \(\gamma_{1,3,2}\), and \(\gamma_{1,3,3}\), and further represent \(\gamma_{1,2,1}\) and \(\gamma_{1,3,1}\) with \(\gamma_{1,1,1}\), \(\gamma_{1,2,2}\) and \(\gamma_{1,3,2}\) with \(\gamma_{1,1,2}\), and \(\gamma_{1,2,3}\) and \(\gamma_{1,3,3}\) with \(\gamma_{1,1,3}\), we can ultimately use \(\gamma_{1,1,1}\), \(\gamma_{1,1,2}\), and \(\gamma_{1,1,3}\) to represent all remaining \(\gamma\) terms. Therefore, in the formulas, we replace all other \(\gamma\) terms with \(\gamma_{1,1,1}\), \(\gamma_{1,1,2}\), and \(\gamma_{1,1,3}\).

\scalebox{0.58}{%
$
 \frac{\gamma_{1,1,1} ( \lambda_{1,1,1}^{3}+\frac{\lambda_{2,1,1}^{3}\lambda_{1,1,1}^{1}}{\lambda_{2,1,1}^{1}k^1_1} +\frac{\lambda_{3,1,1}^{3}\lambda_{1,1,1}^{1}}{\lambda_{3,1,1}^{1} k_2^1}+ \frac{\lambda_{1,1,1}^{2}+\frac{\lambda_{2,1,1}^{2}\lambda_{1,1,1}^{1}}{ \lambda_{2,1,1}^{1} k^{1}_{1}}+\frac{\lambda_{3,1,1}^{2}\lambda_{1,1,1}^{1}}{\lambda_{3,1,1}^{1}k^{1}_{2}}}{(\lambda_{1,2,1}^{2}+\frac{\lambda_{2,1,1}^{2}\lambda_{1,2,1}^{1}}{\lambda_{2,2,1}^{1}k^{1}_{1}}+\frac{\lambda_{3,2,1}^{2}\lambda_{1,2,1}^{1}}{\lambda_{3,2,1}^{1}k^{1}_{2}})k_1^2}(\lambda_{1,2,1}^{3}+\frac{\lambda_{2,2,1}^{3}\lambda_{1,2,1}^{1}}{\lambda_{2,2,1}^{1} k_1^1}+\frac{\lambda_{3,2,1}^{3}\lambda_{1,2,1}^{1}}{\lambda_{3,2,1}^{1} k_2^1})+\frac{\lambda_{1,1,1}^{2}+\frac{\lambda_{2,1,1}^{2}\lambda_{1,1,1}^{1}}{ \lambda_{2,1,1}^{1} k^{1}_{1}}+\frac{\lambda_{3,1,1}^{2}\lambda_{1,1,1}^{1}}{\lambda_{3,1,1}^{1}k^{1}_{2}}}{(\lambda_{1,3,1}^{2}+\frac{\lambda_{2,3,1}^{2}\lambda_{1,3,1}^{1}}{\lambda_{2,3,1}^{1}k^{1}_{1}}+\frac{\lambda_{3,3,1}^{2}\lambda_{1,3,1}^{1}}{\lambda_{3,3,1}^{1}k^{1}_{2}})k_2^2}(\lambda_{1,3,1}^{3}+\frac{\lambda_{2,3,1}^{3}\lambda_{1,3,1}^{1}}{\lambda_{2,3,1}^{1} k_1^1}+\frac{\lambda_{3,3,1}^{3}\lambda_{1,3,1}^{1}}{\lambda_{3,3,1}^{1} k_2^1}) )}{\gamma_{1,1,2} (\lambda_{1,1,2}^{3}+\frac{\lambda_{2,1,2}^{3}\lambda_{1,1,2}^{1}}{\lambda_{2,1,2}^{1}k_1^1}+\frac{\lambda_{3,1,2}^{3}\lambda_{1,1,2}^{1}}{\lambda_{3,1,2}^{1}k_2^1}+\frac{\lambda_{1,1,2}^{2}+\frac{\lambda_{2,1,2}^{2}\lambda_{1,1,2}^{1}}{\lambda_{2,1,2}^{1} k^{1}_{1}}+\frac{\lambda_{3,1,2}^{2}\lambda_{1,1,2}^{1}}{\lambda_{3,1,2}^{1}k^{1}_{2}}}{(\lambda_{1,2,2}^{2}+\frac{\lambda_{2,2,2}^{2}\lambda_{1,2,2}^{1}}{\lambda_{2,2,2}^{1}k^{1}_{1}}+\frac{\lambda_{3,2,2}^{2}\lambda_{1,2,2}^{1}}{\lambda_{3,2,2}^{1}k^{1}_{2}})k_1^2}(\lambda_{1,2,2}^{3}+\frac{\lambda_{2,2,2}^{3}\lambda_{1,2,2}^{1}}{\lambda_{2,2,2}^{1}k_1^1}+\frac{\lambda_{3,2,2}^{3}\lambda_{1,2,2}^{1}}{\lambda_{3,2,2}^{1} k_2^1})+\frac{\lambda_{1,1,2}^{2}+\frac{\lambda_{2,1,2}^{2}\lambda_{1,1,2}^{1}}{\lambda_{2,1,2}^{1} k^{1}_{1}}+\frac{\lambda_{3,1,2}^{2}\lambda_{1,1,2}^{1}}{\lambda_{3,1,2}^{1}k^{1}_{2}}}{(\lambda_{1,3,2}^{2}+\frac{\lambda_{2,3,2}^{2}\lambda_{1,3,2}^{1}}{\lambda_{2,3,2}^{1}k^{1}_{1}}+\frac{\lambda_{3,3,2}^{2}\lambda_{1,3,2}^{1}}{\lambda_{3,3,2}^{1}k^{1}_{2}})k_2^2}(\lambda_{1,3,2}^{3}+\frac{\lambda_{2,3,2}^{3}\lambda_{1,3,2}^{1}}{\lambda_{2,3,2}^{1}k_1^1}+\frac{\lambda_{3,3,2}^{3}\lambda_{1,3,2}^{1}}{\lambda_{3,3,2}^{1}k_2^1}))}=k^3_1
$
 }

\scalebox{0.58}{%

$
\frac{\gamma_{1,1,1} ( \lambda_{1,1,1}^{3}+\frac{\lambda_{2,1,1}^{3}\lambda_{1,1,1}^{1}}{\lambda_{2,1,1}^{1}k^1_1} +\frac{\lambda_{3,1,1}^{3}\lambda_{1,1,1}^{1}}{\lambda_{3,1,1}^{1} k_2^1}+ \frac{\lambda_{1,1,1}^{2}+\frac{\lambda_{2,1,1}^{2}\lambda_{1,1,1}^{1}}{ \lambda_{2,1,1}^{1} k^{1}_{1}}+\frac{\lambda_{3,1,1}^{2}\lambda_{1,1,1}^{1}}{\lambda_{3,1,1}^{1}k^{1}_{2}}}{(\lambda_{1,2,1}^{2}+\frac{\lambda_{2,1,1}^{2}\lambda_{1,2,1}^{1}}{\lambda_{2,2,1}^{1}k^{1}_{1}}+\frac{\lambda_{3,2,1}^{2}\lambda_{1,2,1}^{1}}{\lambda_{3,2,1}^{1}k^{1}_{2}})k_1^2}(\lambda_{1,2,1}^{3}+\frac{\lambda_{2,2,1}^{3}\lambda_{1,2,1}^{1}}{\lambda_{2,2,1}^{1} k_1^1}+\frac{\lambda_{3,2,1}^{3}\lambda_{1,2,1}^{1}}{\lambda_{3,2,1}^{1} k_2^1})+\frac{\lambda_{1,1,1}^{2}+\frac{\lambda_{2,1,1}^{2}\lambda_{1,1,1}^{1}}{ \lambda_{2,1,1}^{1} k^{1}_{1}}+\frac{\lambda_{3,1,1}^{2}\lambda_{1,1,1}^{1}}{\lambda_{3,1,1}^{1}k^{1}_{2}}}{(\lambda_{1,3,1}^{2}+\frac{\lambda_{2,3,1}^{2}\lambda_{1,3,1}^{1}}{\lambda_{2,3,1}^{1}k^{1}_{1}}+\frac{\lambda_{3,3,1}^{2}\lambda_{1,3,1}^{1}}{\lambda_{3,3,1}^{1}k^{1}_{2}})k_2^2}(\lambda_{1,3,1}^{3}+\frac{\lambda_{2,3,1}^{3}\lambda_{1,3,1}^{1}}{\lambda_{2,3,1}^{1} k_1^1}+\frac{\lambda_{3,3,1}^{3}\lambda_{1,3,1}^{1}}{\lambda_{3,3,1}^{1} k_2^1}) )}{\gamma_{1,1,3}(\lambda_{1,1,3}^{3}+\frac{\lambda_{2,1,3}^{3}\lambda_{1,1,3}^{1}}{\lambda_{2,1,3}^{1}k_1^1}+\frac{\lambda_{3,1,3}^{3}\lambda_{1,1,3}^{1}}{\lambda_{3,1,3}^{1}k_2^1}+\frac{\lambda_{1,1,3}^{2}+\frac{\lambda_{2,1,3}^{2}\lambda_{1,1,3}^{1}}{\lambda_{2,1,3}^{1} k^{1}_{1}}+\frac{\lambda_{3,1,3}^{2}\lambda_{1,1,3}^{1}}{\lambda_{3,1,3}^{1}k^{1}_{2}}}{(\lambda_{1,2,3}^{2}+\frac{\lambda_{2,2,3}^{2}\lambda_{1,2,3}^{1}}{\lambda_{2,2,3}^{1}k^{1}_{1}}+\frac{\lambda_{3,2,3}^{2}\lambda_{1,2,3}^{1}}{\lambda_{3,2,3}^{1}k^{1}_{2}})k_1^2}(\lambda_{1,2,3}^{3}+\frac{\lambda_{2,2,3}^{3}\lambda_{1,2,3}^{1}}{\lambda_{2,2,3}^{1}k_1^1}+\frac{\lambda_{3,2,3}^{3}\lambda_{1,2,3}^{1}}{\lambda_{3,2,3}^{1}k_2^1})+\frac{\lambda_{1,1,3}^{2}+\frac{\lambda_{2,1,3}^{2}\lambda_{1,1,3}^{1}}{\lambda_{2,1,3}^{1} k^{1}_{1}}+\frac{\lambda_{3,1,3}^{2}\lambda_{1,1,3}^{1}}{\lambda_{3,1,3}^{1}k^{1}_{2}}}{(\lambda_{1,3,3}^{2}+\frac{\lambda_{2,3,3}^{2}\lambda_{1,3,3}^{1}}{\lambda_{2,3,3}^{1}k^{1}_{1}}+\frac{\lambda_{3,3,3}^{2}\lambda_{1,3,3}^{1}}{\lambda_{3,3,3}^{1}k^{1}_{2}})k_2^2}(\lambda_{1,3,3}^{3}+\frac{\lambda_{2,2,3}^{3}\lambda_{1,3,3}^{1}}{\lambda_{2,3,3}^{1}k_1^1}+\frac{\lambda_{3,3,3}^{3}\lambda_{1,3,3}^{1}}{\lambda_{3,3,3}^{1}k_2^1}))} = k^3_2 $}

Upon organizing the formula, we find that \(\gamma_{1,1,1}\) can be used to represent \(\gamma_{1,1,2}\) and \(\gamma_{1,1,3}\), and all the coefficients involved in this representation are positive. Consequently, \(\gamma_{1,1,1}\) can be utilized to represent all the remaining  \(26\gamma\) terms, with all associated coefficients being positive.

According to the established proportional relationships between the \(\gamma\) terms, assigning any positive value to \(\gamma_{1,1,1}\) will satisfy the six proportional relations \(k^{1}_{1}\), \(k^{1}_{2}\), \(k^{2}_{1}\), \(k^{2}_{2}\), \(k^{3}_{1}\), and \(k^{3}_{2}\). This suggests that a robust mathematical model has been constructed, allowing the proportional relationships of the entire system to be controlled by adjusting the value of \(\gamma_{1,1,1}\). To validate this model, specific code can be executed to check the accuracy of these relationships.

If you would like to see the numerical verification of the above example, please contact us via email at $suweidong @ sdfmu.edu.cn$. We are happy to provide the code.

\begin{theorem}
If a real nonsquare matrix $A \in \mathbb{R}^{m \times n}$ is an individual Volterra-Lyapunov stable matrix (see Definition \ref{def_NSQ_lyp_ind}),
then, there exists a
 block diagonal non-square matrix $K \in \mathbb{C}^{n \times m}$ such that for all non-negative diagonal matrices $E \in  \mathbb{R}^{n \times n}$, and $j \in \mathcal{M}$, where $\mathcal{M}$ is the index set consisting of $k$ tuples of integers in the
range $1,\cdots, m$,
\begin{equation}\label{suf_con_4}
Re\{ \sigma_i ( [A E K]_j ) \} >0,
\end{equation}
where $\text{Re} \{ \sigma_i (M) \}$ represents the real part of the $i$-th eigenvalue of matrix $M$.
\end{theorem}

\begin{proof}
Considering that $A$ is an individual Volterra-Lyapunov stable matrix, the following inequality should be true:
\begin{equation}\label{eq_mdl}
[A]^m_{s_i} D_i + D_i ([A]^m_{s_i})^T  > 0.
\end{equation}
Assume the total number of the $m_{th}$-order squared matrices $[A]^m_{s_i}$ is $N$, i.e., $i \in \{1,2, \cdots, N\}$. Based on inequality (\ref{eq_mdl}), we can conclude, for any $\gamma_i>0$, $i \in \{1, \cdots, N\}$, the following inequality holds
\begin{equation} \label{al_1}
\begin{aligned}
(\gamma_1 [A]^m_{s_1} D^m_{s_1} + \gamma_2 [A]^m_{s_2} D^m_{s_2} \cdots + \gamma_N [A]^m_{s_N} D^m_{s_N}) & +\\
  ( D^m_{s_1}\gamma_1 [A]^m_{s_1} + D^m_{s_2}\gamma_2 [A]^m_{s_2} \cdots + D^m_{s_N} \gamma_N [A]^m_{s_N})^T & >0
\end{aligned}
\end{equation}

Based on Lemma \ref{lemma2}, for any diagonal positive matrix $D > 0$ and $j \in \mathcal{M}$, we have
\begin{equation}\label{eq_important}
Re\{ \sigma  ( [\sum_{i=1}^{N} \gamma_i  [A]^m_{s_i}D ]_j ) \} >0.
\end{equation}

Without loss of generality, we only consider the proof for the $m_{th}$-order matrices. So, we omit the
subscript $j$ in the following discussions.

According to structure of $K$, as before, we redefine the sub-indices of the $n$ columns of $A$ as follows:
$$A=[\boldsymbol{a}_{1,1}, \cdots, \boldsymbol{a}_{1,p_1}, \boldsymbol{a}_{2,1}, \cdots, \boldsymbol{a}_{2,p_2}, \cdots, \cdots, \boldsymbol{a}_{m,1} \cdots,\boldsymbol{a}_{m,p_m}].$$

Correspondingly, we redefine the sub-indices of $\gamma_i>0$, $i \in \{1, \cdots, N\}$ in inequality (\ref{al_1}) as follows;
\begin{equation}
\gamma_{\kappa_1 \kappa_2 \cdots \kappa_m},
\end{equation}
 The coefficient $\gamma_{\kappa_1 \kappa_2 \cdots \kappa_m}$ means from each of the following group of vectors to select one vector to correspond to a term $\gamma_{\kappa_1 \kappa_2 \cdots \kappa_m}[A]^m_{s_i}$ in inequality  (\ref{eq_important}):
\begin{equation}
\begin{aligned}
 Group\,1: & \,\,  \boldsymbol{a}_{1,1}, \cdots, \boldsymbol{a}_{1,p_1},\\
 Group\,2:  & \,\, \boldsymbol{a}_{2,1}, \cdots, \boldsymbol{a}_{2,p_2}, \\
  &\cdots, \cdots, \\
 Group\,m: & \,\, \boldsymbol{a}_{m,1} \cdots,\boldsymbol{a}_{m,p_m}. \\
\end{aligned}
\end{equation}

Further assume $ D=diag\{[d_1,d_2,\cdots, d_m]\}$ and $ D^m_{s_i} = diag\{[\lambda^1_i,\lambda^2_i,\cdots, \lambda^m_i]\}$.  Then the core part of the equation (\ref{eq_important}) can be expressed as follows:
\begin{equation} \label{carefu}
\begin{aligned}
     (\sum_{i=1}^{N} \gamma_i [A]^m_{s_i}  D^m_{s_i})
=[& d_1 ( \sum_{j=1}^{p_1} (\sum_{\kappa_2=1}^{p_2} \sum_{\kappa_3=1}^{p_3} \cdots \sum_{\kappa_m}^{p_m} \gamma_{ j,\kappa_2, \cdots, \kappa_m } \lambda^1_{j,\kappa_2, \cdots, \kappa_m } ) \boldsymbol{a}_{1,j} ), \\
     & d_2 ( \sum_{j=1}^{p_2} (\sum_{\kappa_1=1}^{p_1} \sum_{\kappa_3=1}^{p_3} \cdots \sum_{\kappa_m}^{p_m} \gamma_{ \kappa_1, j, \kappa_3 ,\cdots, \kappa_m } \lambda^2_{\kappa_1, j, \kappa_3 ,\cdots, \kappa_m}) \boldsymbol{a}_{2,j} ), \\
     & \cdots \cdots \cdots \\
      & d_m  ( \sum_{j=1}^{p_m} (\sum_{\kappa_1=1}^{p_1} \sum_{\kappa_2=1}^{p_2} \cdots \sum_{\kappa_{(m-1)}}^{p_{(m-1)}} \gamma_{ \kappa_1, \kappa_2, \cdots,  \kappa_{(m-1)} ,j} \lambda^m_{\kappa_1, \kappa_2, \cdots,  \kappa_{(m-1)} ,j} ) \boldsymbol{a}_{m,j} ) ].\\
\end{aligned}
\end{equation}

On the other hand, the matrix $AEK$ can be expressed as
\begin{equation}\label{eq_li_comb1}
\begin{aligned}
AEK &= [ \sum_{j=1}^{p_1} \varepsilon_{1j} k_{1j} \boldsymbol{a}_{1,j}, \sum_{j=1}^{p_2} \varepsilon_{2j} k_{2j} \boldsymbol{a}_{2,j}, \cdots, \sum_{j=1}^{p_m} \varepsilon_{mj} k_{mj} \boldsymbol{a}_{m,j}]\\
& = [ \sum_{j=1}^{p_1} \tilde {k}_{1j} \boldsymbol{a}_{1,j}, \sum_{j=1}^{p_2} \tilde {k}_{2j} \boldsymbol{a}_{2,j}, \cdots, \sum_{j=1}^{p_m} \tilde {k}_{mj} \boldsymbol{a}_{m,j}]\\
& = [ \tilde {k}_{11} \sum_{j=1}^{p_1} \frac{\tilde {k}_{1j}}{\tilde {k}_{11}}  \boldsymbol{a}_{1,j}, \tilde {k}_{21} \sum_{j=1}^{p_2} \frac{\tilde {k}_{2j}}{\tilde {k}_{21}}  \boldsymbol{a}_{2,j}, \cdots, \tilde {k}_{m1} \sum_{j=1}^{p_m} \frac{\tilde {k}_{mj}}{\tilde {k}_{m1}} \boldsymbol{a}_{m,j}]\\
& = [ \tilde {k}_{11} \sum_{j=1}^{p_1} \bar {k}_{1j} \boldsymbol{a}_{1,j}, \tilde {k}_{21} \sum_{j=1}^{p_2} \bar {k}_{2j} \boldsymbol{a}_{2,j}, \cdots, \tilde {k}_{m1} \sum_{j=1}^{p_m} \bar {k}_{mj} \boldsymbol{a}_{m,j}].\\
\end{aligned}
\end{equation}
In the above equation, $\tilde {k}_{ij} = \varepsilon_{i,j} k_{i,j} \ge 0$, and  $\bar {k}_{ij} = \frac{\tilde {k}_{i,j}}{\tilde {k}_{i,1}}$ (considering the case $\tilde {k}_{i,1} \ne 0$), which implies $\bar {k}_{i1}=1$.

Now, according to Lemma \ref{pailiezuhe_Lemma}, in Equation (\ref{carefu}), if we let $d_i = \tilde{k}_{i1}$ ($i \in \{1,2,\cdots,m\}$), and
 for any given \( k^{\delta}_{\zeta} \) in the \( \delta^{th} \) group and \( \zeta^{th} \) proportion (where \( 1 \leq \delta \leq m \), \( 1 \leq \zeta \leq p_\delta-1 \), and \( \delta \), \( \zeta \) are positive integers), the following equation relationship exists:

\scalebox{0.58}{%

$\frac{\overbrace{\overbrace{\lambda^{\delta}_{\underbrace{\scriptstyle1,1,\ldots,1}_{m}} \gamma_{\underbrace{\scriptstyle1,1,\ldots,1}_{m}} +\ldots+ \lambda^{\delta}_{\underbrace{\scriptstyle1,1,\ldots,1,1}_{\delta},p_{\delta+1},p_{\delta + 2},\ldots, p_{m-1},p_m} \gamma_{\underbrace{\scriptstyle1,1,\ldots,1,1}_{\delta},p_{\delta+1},p_{\delta + 2},\ldots, p_{m-1},p_m} }^{p_{\delta+1}p_{\delta +2}\cdots p_m} + \overbrace{ \ldots+\lambda^{\delta}_{\underbrace{\scriptstyle p_1,\ldots, p_{\delta-1},1,}_{\delta}\underbrace{\scriptstyle p_{\delta+1},\ldots, p_m}_{m-\delta}} \gamma_{\underbrace{\scriptstyle p_1,\ldots, p_{\delta-1},1,}_{\delta}\underbrace{\scriptstyle p_{\delta+1},\ldots, p_m}_{m-\delta}}}^{p_1p_2\cdots p_{\delta-1}p_{\delta+1}\cdots p_m-p_{\delta+1}p_{\delta +2}\cdots p_m}}^{p_1p_2\cdots p_{\delta-1}p_{\delta+1}\cdots p_m}}{\overbrace{\overbrace{\lambda^{\delta}_{\underbrace{\scriptstyle1,1,\ldots, 1,(\zeta+1)}_{\delta}\underbrace{\scriptstyle1,\ldots, 1}_{m-\delta}} \gamma_{\underbrace{\scriptstyle1,1,\ldots, 1,(\zeta+1)}_{\delta}\underbrace{\scriptstyle1,\ldots, 1}_{m-\delta}} +\ldots +\lambda^{\delta}_{\underbrace{\scriptstyle1,1,\ldots, 1,(\zeta+1)}_{\delta}\underbrace{\scriptstyle p_{\delta+1}p_{\delta+2}\ldots p_m}_{m-\delta}} \gamma_{\underbrace{\scriptstyle1,1,\ldots, 1,(\zeta+1)}_{\delta}\underbrace{\scriptstyle p_{\delta+1}p_{\delta+2}\ldots p_m}_{m-\delta}} }^{p_{\delta+1}p_{\delta +2}\cdots p_m}+\overbrace{\ldots +\lambda^{\delta}_{\underbrace{\scriptstyle p_1,\ldots, p_{\delta-1},(\zeta+1),}_{\delta}\underbrace{\scriptstyle p_{\delta+1},\ldots, p_m}_{m-\delta}} \gamma_{\underbrace{\scriptstyle p_1,\ldots, p_{\delta-1},(\zeta+1),}_{\delta}\underbrace{\scriptstyle p_{\delta+1},\ldots, p_m}_{m-\delta}}}^{p_1p_2\cdots p_{\delta-1}p_{\delta+1}\cdots p_m-p_{\delta+1}p_{\delta +2}\cdots p_m}}^{p_1p_2\cdots p_{\delta-1}p_{\delta+1}\cdots p_m}}= k^{\delta}_{\zeta}.$
}

According to the proof of Lemma \ref{pailiezuhe_Lemma}, the product terms of $\lambda$ and $\gamma$ corresponding to the numerator and denominator are proportional to $k^{\delta}_{\zeta}$,
the left-hand side of Equation (\ref{carefu}) equals that of Equation (\ref{eq_li_comb1}). Thus, we can conclude, that for any non-negative diagonal matrices $E \in  \mathbb{R}^{n \times n}$, the inequality (\ref{suf_con_1}) is always satisfied.
\end{proof}

\section{CONCLUSION}
In this study, we consider extended D-stability and Volterra-Lyapunov stability for non-square matrices and establish links between these properties for both square and non-square matrices. The main result is a sufficient condition for D-stability of non-square matrices. Specifically, we confirm that if a non-square matrix is individually Volterra-Lyapunov stable, then it is also D-stable. This condition is rather mild, as it remains an open problem to identify the necessary and sufficient condition for D-stability even for square matrices. 

Furthermore, in some decentralized control or optimization problems, there are different options for selecting the structure of the block diagonal matrix \(K\). Therefore, investigating the selection of the structure of the block diagonal matrix \(K\) to satisfy the D-stable condition would be an interesting direction for further research.

\appendix
\section{Proof of Lemma 2}

\begin{proof}
    
Consider that we have $m$ groups, and let $\phi$ denote the current group index, where $0 < \phi \leq m$. Each group $\phi$ contains $p_\phi$ distinct cards. Our objective is to traverse each card in each group in a non-repeating combination manner. Let $\gamma$ denote a specific way of combination. It is calculated that the total number of such combinations is $\prod_{\phi=1}^m p_\phi$, denoted by $\gamma_{1,1,\ldots,1}$ up to $\gamma_{p_1,p_2,\ldots,p_m}$. Consequently, there are $m \prod_{\phi=1}^m p_\phi$ parameters $\lambda$ obtained through this calculation. Furthermore, the ratio denoted by $k$ has a total of $\sum_{\phi=1}^{m} (p_\phi - 1)$ instances, represented by $k^{1}_{1}$ up to $k^{1}_{p_1}, \ldots, k^{m}_{1}$ up to $k^{m}_{p_m}$.

For the $p_1-1$ ratio relationships from $k^{1}_{1}$ to $k^{1}_{p_1-1}$, $p_1-1$ equations can be established.

\scalebox{0.54}{%

$\frac{\overbrace{\overbrace{\lambda_{\underbrace{\scriptstyle1,1,\ldots,1,1}_{m}}^{1} \gamma_{\underbrace{\scriptstyle1,1,\ldots,1,1}_{m}} + \lambda_{\underbrace{\scriptstyle1,1,\ldots,1,2}_{m}}^{1} \gamma_{\underbrace{\scriptstyle1,1,\ldots,1,2}_{m}} + \ldots + \lambda_{\underbrace{\scriptstyle1,1,\ldots,1,{p_m-1}}_{m}}^{1} \gamma_{\underbrace{\scriptstyle1,1,\ldots,1,{p_m-1}}_{m}} + \lambda_{\underbrace{\scriptstyle1,1,\ldots,1,p_m}_{m}}^{1}\gamma_{\underbrace{\scriptstyle1,1,\ldots,1,p_m}_{m}}}^{p_m} + \lambda_{\underbrace{\scriptstyle1,1,\ldots,2,1}_{m}}^{1} \gamma_{\underbrace{\scriptstyle1,1,\ldots,2,1}_{m}} +\ldots+\lambda_{\underbrace{\scriptstyle1,p_2,\ldots, p_{m-1},p_m}_{m}}^{1} \gamma_{\underbrace{\scriptstyle1,p_2,\ldots, p_{m-1},p_m}_{m}}}^{p_2 p_3 \cdots p_m}}{\overbrace{\overbrace{\lambda_{\underbrace{\scriptstyle2,1,\ldots,1,1}_{m}}^{1}\gamma_{\underbrace{\scriptstyle2,1,\ldots,1,1}_{m}} + \lambda_{\underbrace{\scriptstyle2,1,\ldots,1,2}_{m}}^{1} \gamma_{\underbrace{\scriptstyle2,1,\ldots,1,2}_{m}} + \ldots+\lambda_{\underbrace{\scriptstyle2,1,\ldots,1,{p_m-1}}_{m}}^{1} \gamma_{\underbrace{\scriptstyle2,1,\ldots,1,{p_m-1}}_{m}} + \lambda_{\underbrace{\scriptstyle2,1,\ldots,1,p_m}_{m}}^{1} \gamma_{\underbrace{\scriptstyle2,1,\ldots,1,p_m}_{m}}}^{p_m}+ \lambda_{\underbrace{\scriptstyle2,1,\ldots,1,2,1}_{m}}^{1} \gamma_{\underbrace{\scriptstyle2,1,\ldots,1,2,1}_{m}} + \ldots+\lambda_{\underbrace{\scriptstyle2,p_2,\ldots,p_m}_{m}}^{1} \gamma_{\underbrace{\scriptstyle2,p_2,\ldots,p_m}_{m}}}^{p_2 p_3 \cdots p_m}} = k^{1}_{1}
$}

\scalebox{0.54}{%

$\frac{\overbrace{\overbrace{\lambda_{\underbrace{\scriptstyle1,1,\ldots,1,1}_{m}}^{1} \gamma_{\underbrace{\scriptstyle1,1,\ldots,1,1}_{m}} + \lambda_{\underbrace{\scriptstyle1,1,\ldots,1,2}_{m}}^{1} \gamma_{\underbrace{\scriptstyle1,1,\ldots,1,2}_{m}} + \ldots + \lambda_{\underbrace{\scriptstyle1,1,\ldots,1,{p_m-1}}_{m}}^{1} \gamma_{\underbrace{\scriptstyle1,1,\ldots,1,{p_m-1}}_{m}} + \lambda_{\underbrace{\scriptstyle1,1,\ldots,1,p_m}_{m}}^{1}\gamma_{\underbrace{\scriptstyle1,1,\ldots,1,p_m}_{m}}}^{p_m} + \lambda_{\underbrace{\scriptstyle1,1,\ldots,2,1}_{m}}^{1} \gamma_{\underbrace{\scriptstyle1,1,\ldots,2,1}_{m}} +\ldots+\lambda_{\underbrace{\scriptstyle1,p_2,\ldots, p_{m-1},p_m}_{m}}^{1} \gamma_{\underbrace{\scriptstyle1,p_2,\ldots, p_{m-1},p_m}_{m}}}^{p_2 p_3 \cdots p_m}}{\overbrace{\overbrace{\lambda_{\underbrace{\scriptstyle p_1,1,\ldots, 1,1}_{m}}^{1} \gamma_{\underbrace{\scriptstyle p_1,1,\ldots, 1,1}_{m}} + \ldots + \lambda_{\underbrace{\scriptstyle p_1,1,\ldots,1,p_m}_{m}}^{1} \gamma_{\underbrace{\scriptstyle p_1,1,\ldots,1,p_m}_{m}}}^{p_m}+ \lambda_{\underbrace{\scriptstyle p_1,1,\ldots,1,2,1}_{m}}^{1} \gamma_{\underbrace{\scriptstyle p_1,1,\ldots,1,2,1}_{m}}+\ldots + \lambda_{\underbrace{\scriptstyle p_1p_2\ldots p_m}_{m}}^{1} \gamma_{\underbrace{\scriptstyle p_1p_2\ldots p_m}_{m}}}^{p_2 p_3 \cdots p_m}} = k^{1}_{p_1-1}
$}

These two equations respectively demonstrate the compositional relationships of the ratio $k^{1}_{1}$ between the second element and the first element in the first group, and the compositional relationships of the ratio $k^{1}_{p_1-1}$ between the last element and the first element. Not only are the compositional relationships of these two ratios $k$ presented, but also for each $k$ ratio from $k^{1}_{1}$ to $k^{1}_{p_1-1}$ in this group, the total number of numerators and denominators in the equation

Here, we first deduce the arrangement and distribution of $\lambda$ and $\gamma$ in the $p_1-1$ ratio relationships from $k^{1}_{1}$ to $k^{1}_{p_1-1}$.

For the numerator, the order of arrangement for $\gamma$ remains constant across the $p_1-1$ ratios $k$. For each of these $p_1-1$ ratios $k$, the $\gamma$ called in the product of the first $p_m$ $\lambda$ and $\gamma$ is respectively $\gamma_{1,1,\ldots,1,1}$ up to $\gamma_{1,1,\ldots,1,p_m}$, then the products from $p_m+1$ to $2p_m$ are with $\gamma_{1,1,\ldots,2,1}$ up to $\gamma_{1,1,\ldots,2,p_m}$, and so on, until it can be inferred that the products from $p_2 p_3\cdots p_m-p_m+1$ to $p_2 p_3\cdots p_m$ are with $\gamma_{1,p_2,\ldots, p_{(m-1)},1}$ up to $\gamma_{1,p_2,\ldots, p_{(m-1)},p_m}$. In this way, each formula constituting each $k$ ratio needs to traversing through $p_2p_3\cdots p_m\gamma$.

For the numerator, the order of arrangement for $\lambda$ remains constant across the $p_1-1$ ratios $k$. For each of these $p_1-1$ ratios $k$, the $\gamma$ called in the product of the first $p_m$ $\lambda$ and $\gamma$ is respectively $\lambda_{1,1,\ldots,1,1}^{1}$ up to $\lambda_{1,1,\ldots,1,p_m}^{1}$, then the products from $p_m+1$ to $2p_m$ are with $\lambda_{1,1,\ldots,2,1}^{1}$ up to $\lambda_{1,1,\ldots,2,p_m}^{1}$, and so on, until it can be inferred that the products from $p_2 p_3\cdots p_m-p_m+1$ to $p_2 p_3\cdots p_m$ are with $\lambda_{1,p_2,\ldots, p_{(m-1)},1}^{1}$ up to $\lambda_{1,p_2,\ldots, p_{(m-1)},p_m}^{1}$. In this way, each formula constituting each $k$ ratio needs to traversing through $p_2p_3\cdots p_m\lambda$.

For the denominator, concerning the ratio $k^{1}_{1}$, $\gamma$ is called from $\gamma_{2,1,\ldots,1,1}$ to $\gamma_{2,p_2,\ldots, p_{(m-1)},p_{m}}$; for the ratio $k^{1}_{2}$, $\gamma$ is called from $\gamma_{3,1,\ldots,1,1}$ to $\gamma_{3,p_2,\ldots,p_{(m-1)},p_m}$; for the ratio $k^{ 1}_{p_1-1}$, $\gamma$ is called from $\gamma_{p_1,1,\ldots,1,1}$ to $\gamma_{p_1,p_2,\ldots, p_{(m-1)},p_m}$. For each formula constituting each $k$ ratio, each formula's denominator needs to traverse through $p_2p_3\cdots p_m$ $\gamma$.

For the denominator, concerning the ratio $k^{1}_{1}$, $\lambda$ is called from $\lambda_{2,1,\ldots,1,1}^{1}$ to $\lambda^{1}_{2,p_2,\ldots, p_{(m-1)},p_{m}}$; for the ratio $k^{1}_{2}$, $\lambda$ is called from $\lambda_{3,1,\ldots,1,1}^{1}$ to $\lambda_{3,p_2,\ldots,p_{(m-1)},p_m}^{1}$; for the ratio $k^{ 1}_{p_1-1}$, $\lambda$ is called from $\lambda_{p_1,1,\ldots,1,1}^{1}$ to $\lambda_{p_1,p_2,\ldots, p_{(m-1)},p_m}^{1}$.  For each formula constituting each $k$ ratio, each formula's denominator needs to traverse through $p_2p_3\cdots p_m$ $\lambda$.

For the first $p_1-1$ values of $k$, the same processing method is adopted to solve this problem. For these $p_1-1$ equations, each equation establishes a corresponding proportional relationship. For the first proportion, denoted as $k^{1}_{1}$, both the numerator and the denominator iterate over $p_2p_3\cdots p_m$ products of $\lambda$ and $\gamma$. Therefore, both the numerator and denominator are iterations over $p_2p_3\cdots p_m$ products. Arrange the products of $\gamma$ and constants in the numerator and denominator in ascending order of the subscripts of $\gamma$, with each term corresponding one-to-one. Let
$$
\frac{\lambda_{1,1,\ldots,1,1}^{1} \gamma_{1,1,\ldots,1,1}}{\lambda_{2,1,\ldots,1,1}^{1} \gamma_{2,1,\ldots,1,1}} = k^{1}_{1}, \quad \frac{\lambda_{1,1,\ldots,1,2}^{1} \gamma_{1,1,\ldots,1,2}}{\lambda_{2,1,\ldots,1,2}^{1} \gamma_{2,1,\ldots,1,2}} =  k^{1}_{1}, \ldots, \frac{\lambda_{1,p_1,\ldots, p_m}^{1} \gamma_{1,p_1,\ldots, p_m}}{\lambda_{2,p_2,\ldots ,p_m}^{1} \gamma_{2,p_2,\ldots ,p_m}} =  k^{1}_{1}
$$
for these $p_2 p_3 \cdots p_m $ terms. For the last proportion, which is the $p_1-1^{th}$ one, let

\scalebox{0.88}{%
$\frac{\lambda_{1,1,\ldots,1,1}^{1} \gamma_{1,1,\ldots,1,1}}{\lambda_{p_1,1,\ldots,1,1}^{1} \gamma_{p_1,1,\ldots,1,1}} = k^{1}_{p_1-1}, \frac{\lambda_{1,1,\ldots,1,2}^{1} \gamma_{1,1,\ldots,1,2}}{\lambda_{p_1,1,\ldots,1,2}^{1} \gamma_{p_1,1,\ldots,1,2}} = k^{1}_{p_1-1}, \ldots,\frac{\lambda_{1,p_2\ldots, p_m}^{1} \gamma_{1,p_2\ldots, p_m}}{\lambda_{p_1,p_2,\ldots,p_m}^{1}\gamma_{p_1,p_2,\ldots,p_m}} = k^{1}_{p_1-1}
$}

And for any $\alpha^{th}$ proportion (where $1 \leq \alpha \leq p_1-1$ and $\alpha$ is a positive integer),

\scalebox{0.96}{%
$\frac{\lambda_{1,1,\ldots,1,1}^{1} \gamma_{1,1,\ldots,1,1}}{\lambda_{(\alpha+1),1,\ldots,1,1}^{1} \gamma_{(\alpha+1),1,\ldots,1,1}} = k^{1}_{\alpha }, \frac{\lambda_{1,1,\ldots,1,2}^{1} \gamma_{1,1,\ldots,1,2}}{\lambda_{(\alpha+1),1,\ldots,1,2}^{1} \gamma_{(\alpha+1),1,\ldots,1,2}} = k^{1}_{\alpha }, \ldots, \frac{\lambda_{1,p_2\ldots p_m}^{1} \gamma_{1,p_2\ldots p_m}}{\lambda_{(\alpha+1),p_2\ldots p_m}^{1} \gamma_{(\alpha+1),p_2\ldots p_m}} = k^{1}_{\alpha }
$}

After establishing the proportional relationships of the aforementioned terms, since $\lambda$ and $k$ are known constants, we can construct the proportional relationships between the $\gamma$ terms in the numerator and the $\gamma$ terms in the denominator. This includes the proportions of \(\gamma_{1,1,\ldots,1,1}\) in the numerator to \(\gamma_{2,1,\ldots,1,1}\), \(\gamma_{3,1,\ldots,1,1}\), and so on up to \(\gamma_{p_1,1,\ldots,1,1}\) in the denominator, as well as the proportions of \(\gamma_{1,1,\ldots,1,2}\) in the numerator to \(\gamma_{2,1,\ldots,1,2}\), \(\gamma_{3,1,\ldots,1,2}\), and so on up to \(\gamma_{p_1,1,\ldots,1,2}\) in the denominator, and similarly for other terms up to the proportions of \(\gamma_{1,p_2\ldots p_m}\) in the numerator to \(\gamma_{2,p_2,\ldots, p_m}\), \(\gamma_{3,p_2,\ldots,p_m}\), and so on up to \(\gamma_{p_1,p_2,\ldots p_m}\) in the denominator. In other words, for the first \(p_1-1\) proportional expressions, the \(\gamma\) terms in the denominator can be represented by the corresponding \(\gamma\) terms in the numerator at the same position in the sequence.

For the $p_2-1$ ratio relationships from $k^{2}_{1}$ to $k^{2}_{p_2-1}$, $p_2-1$ equations can be established.

\scalebox{0.66}{%

$\frac{\overbrace{\overbrace{\lambda^{2}_{\underbrace{\scriptstyle1,1,\ldots,1,1}_{m}} \gamma_{\underbrace{\scriptstyle1,1,\ldots,1,1}_{m}} +\ldots+ \lambda^{2}_{\underbrace{\scriptstyle1,1,p_3,\ldots,p_{m-1},p_m}_{m}} \gamma_{\underbrace{\scriptstyle1,1,p_3,\ldots,p_{m-1},p_m}_{m}}}^{p_3 p_4\cdots p_m} +\lambda^{2}_{\underbrace{\scriptstyle2,1,1,\ldots,1,1}_{m}} \gamma_{\underbrace{\scriptstyle2,1,1,\ldots,1,1}_{m}} +  \ldots +\lambda^{2}_{\underbrace{\scriptstyle p_1,1,p_3\ldots p_m}_{m}} \gamma_{\underbrace{\scriptstyle p_1,1,p_3\ldots p_m}_{m}}}^{p_1p_3\cdots p_m}}{\overbrace{\overbrace{\lambda^{2}_{\underbrace{\scriptstyle1,2,1,\ldots,1,1}_{m}} \gamma_{\underbrace{\scriptstyle1,2,1,\ldots,1,1}_{m}} +\ldots+ \lambda^{2}_{\underbrace{\scriptstyle1,2,p_3,\ldots ,p_m}_{m}} \gamma_{\underbrace{\scriptstyle1,2,p_3,\ldots ,p_m}_{m}}}^{p_3p_4\cdots p_m} +  \lambda^{2}_{\underbrace{\scriptstyle2,2,1,\ldots ,1}_{m}} \gamma_{\underbrace{\scriptstyle2,2,1,\ldots ,1}_{m}} +\ldots +\lambda^{2}_{\underbrace{\scriptstyle p_1,2,p_3\ldots p_m}_{m}} \gamma_{\underbrace{\scriptstyle p_1,2,p_3\ldots p_m}_{m}}}^{p_1p_3\cdots p_m}} = k^{2}_{1}
$
}

\scalebox{0.60}{%

$\frac{\overbrace{\overbrace{\lambda^{2}_{\underbrace{\scriptstyle1,1,\ldots,1,1}_{m}} \gamma_{\underbrace{\scriptstyle1,1,\ldots,1,1}_{m}} +\ldots+ \lambda^{2}_{\underbrace{\scriptstyle1,1,p_3,\ldots,p_{m-1},p_m}_{m}} \gamma_{\underbrace{\scriptstyle1,1,p_3,\ldots,p_{m-1},p_m}_{m}}}^{p_3 p_4\cdots p_m} +\lambda^{2}_{\underbrace{\scriptstyle2,1,1,\ldots,1,1}_{m}} \gamma_{\underbrace{\scriptstyle2,1,1,\ldots,1,1}_{m}} +  \ldots +\lambda^{2}_{\underbrace{\scriptstyle p_1,1,p_3\ldots p_m}_{m}} \gamma_{\underbrace{\scriptstyle p_1,1,p_3\ldots p_m}_{m}}}^{p_1p_3\cdots p_m}}{\overbrace{\overbrace{\lambda^{2}_{\underbrace{\scriptstyle1,p_2-1,1,\ldots,1,1}_{m}} \gamma_{\underbrace{\scriptstyle1,p_2-1,1,\ldots,1,1}_{m}} +\ldots+ \lambda^{2}_{\underbrace{\scriptstyle1,p_2-1,p_3,\ldots ,p_m}_{m}}\gamma_{\underbrace{\scriptstyle1,p_2-1,p_3,\ldots ,p_m}_{m}}}^{p_3p_4\cdots p_m} +  \lambda^{2}_{\underbrace{\scriptstyle2,p_2-1,1,\ldots ,1}_{m}} \gamma_{\underbrace{\scriptstyle2,p_2-1,1,\ldots ,1}_{m}} +\ldots +\lambda^{2}_{\underbrace{\scriptstyle p_1,p_2-1,p_3\ldots p_m}_{m}} \gamma_{\underbrace{\scriptstyle p_1,p_2-1,p_3\ldots p_m}_{m}}}^{p_1p_3\cdots p_m}} = k^{2}_{p_2-1}
$
}

These two equations respectively demonstrate the compositional relationships of the ratios \( k^{2}_{1} \) and \( k^{2}_{p_2-1} \) in the second group, specifically between the second and the first elements, and the last and the first elements. Not only these two, but for every \( k \) from \( k^{2}_{1} \) to \( k^{2}_{p_2-1} \) in this group, the number of terms in both the numerator and denominator of each proportional relationship equation is \(p_1p_3\cdots p_m \). Each product of \( \lambda \) and \( \gamma \) in the equations is arranged in sequence as shown in the formula.

Here, we first deduce the arrangement and distribution of \( \lambda \) and \( \gamma \) in the \( p_2-1 \) proportional relationships ranging from \( k^{2}_{1} \) to \( k^{2}_{p_2-1} \). This involves determining the specific sequence in which these variables are ordered within each of the proportional relationships.

For the numerators involving \( \gamma \) in these \( p_2-1 \) proportions denoted as \( k \), the sequence of \( \gamma \)'s remains constant. For each of these \(p_2-1 \) proportions \( k \), the first \(p_m \) products of \( \lambda \) and \( \gamma \) involve \( \gamma \) ranging from \( \gamma_{1,1,\ldots,1,1} \) to \( \gamma_{1,1,\ldots,1,p_m} \). Then, the \( p_m+1 \)st to \( 2p_m \)th products involve \( \gamma \) ranging from \( \gamma_{1,1,\ldots,1,2,1} \) to \( \gamma_{1,1,\ldots,1,2,p_m} \), and so on, until the \( p_3p_4\cdots p_m-p_m+1 \)th to \( p_3p_4\cdots p_m \)th products, which involve \( \gamma \) ranging from \( \gamma_{1,1,p_3,\ldots,p_{m-1},1} \) to \( \gamma_{1,1,p_3,\ldots ,p_{m-1},p_m} \). Notably, the \(p_3p_4\cdots p_m+1  \)st to \( p_3p_4\cdots p_m+p_m \)th products involve \( \gamma \) ranging from \( \gamma_{2,1,1,\ldots,1,1} \) to \( \gamma_{2,1,1,\ldots,1,p_m} \), and this pattern continues up to the last set of products from \( p_1p_3p_4\cdots p_m -p_m+1\)st to \(p_1p_3p_4\cdots p_m \)th, which involve \( \gamma \) ranging from \( \gamma_{p_1,1,p_3,\ldots, p_{m-1},1} \) to \( \gamma_{p_1,1,p_3,\ldots, p_{m-1},p_m} \). Thus, each formula constituting a \( k \) proportion traverses \(p_1p_3\cdots p_m \) \( \gamma \)'s in its numerator.

For the numerator, the order of arrangement for $\lambda$ remains constant across the $p_2-1$ ratios $k$. For each of these $p_2-1$ ratios $k$, the $\gamma$ called in the product of the first $p_m$ $\lambda$ and $\gamma$ is respectively $\lambda^{2}_{1,1,\ldots,1,1}$ up to $\lambda^{2}_{1,1,\ldots,1,p_m}$, then the products from $p_m+1$ to $2p_m$ are with $\lambda^{2}_{1,1,\ldots,1,2,1}$ up to $\lambda^{2}_{1,1,\ldots,1,2,p_m}$, and so on, until it can be inferred that the products from $p_2 p_3\cdots p_m-p_m+1$ to $p_2 p_3\cdots p_m$ are with $\lambda^{2}_{1,1,p_3,\ldots,p_{m-1},1}$ up to $\lambda^{2}_{1,1,p_3,\ldots ,p_{m-1},p_m}$. In this way, each formula constituting each $k$ ratio needs to traversing through $p_1p_3\cdots p_m\lambda$.

For the denominators in the proportions denoted by \( k \), the sequence of \( \gamma \) varies for each proportion. For the proportion \( k^{2}_{1} \), \( \gamma \) is taken from \( \gamma_{1,2,1,\ldots,1,1} \) to \( \gamma_{p_1,2,p_3\ldots p_m} \); for \( k^{2}_{2} \), it is from \( \gamma_{1,3,1\ldots,1,1} \) to \( \gamma_{p_1,3,p_3,\ldots, p_m} \); and for \( k^{2}_{p_2-1} \), \( \gamma \) ranges from \( \gamma_{1,p_2,1,\ldots,1,1} \) to \( \gamma_{p_1,p_2,\ldots,p_m} \). Each formula constituting a \( k \) proportion traverses \(p_1p_3\cdots p_m \) \( \gamma \)'s in its denominator.

For the denominators in the proportions denoted by \( k \), the sequence of \( \lambda \) also varies for each proportion. Specifically, for the proportion \( k^{2}_{1} \), \( \lambda \) ranges from \( \lambda^{2}_{1,2,1,\ldots,1,1} \) to \( \lambda^{2}_{p_1,2,p_3\ldots p_m} \); for \( k^{2}_{2} \), it is from \( \lambda^{2}_{1,3,1\ldots,1,1} \) to \( \lambda^{2}_{p_1,3,p_3,\ldots, p_m} \); and for \( k^{2}_{p_2-1} \), \( \lambda \) ranges from \( \lambda^{2}_{1,p_2,1,\ldots,1,1} \) to \( \lambda^{2}_{p_1,p_2,\ldots,p_m} \). Each formula constituting a \( k \) proportion traverses \( p_1p_3\cdots p_m \) \( \lambda \)'s in its denominator.

Here, for these \( p_2-1 \) proportions \( k \), the same processing method as the previous \( p_1-1 \) \( k \) is adopted to solve the problem. For these \( p_2-1 \) equations, each equation establishes a corresponding proportional relationship. For the first proportion \( k^{2}_{1} \), both the numerator and the denominator iterate over \( p_1p_3\cdots p_m \) products of \( \lambda \) and \( \gamma \), thus each traverses \(p_1p_3\cdots p_m \) products. Arrange the products in the numerator and denominator in ascending order of their subscripts, each corresponding one-to-one. Particularly note the part starting from \( \gamma_{2,1,1,\ldots,1,1} \); based on the reasoning from \( k^{1}_{1} \) to \( k^{1}_{p_1-1} \), \( \gamma_{2,1,1,\ldots,1,1} \) can be expressed using \( \gamma_{1,1,\ldots,1,1} \) and known constants \( \lambda \) and \( k^1 \), i.e., \( \gamma_{2,1,\ldots,1,1} = \frac{\lambda_{1,1,\ldots,1,1}^{1} \gamma_{1,1,\ldots,1,1}}{\lambda_{2,1,\ldots,1,1}^{1} k^{1}_{1}} \) where \( \frac{\lambda_{1,1,\ldots,1,1}^{1}}{\lambda_{2,1,\ldots,1,1}^{1} k^{1}_{1}} \) is a known constant, thus transforming \( \gamma_{2,1,1,\ldots,1,1} \) into a product of \( \gamma_{1,1,\ldots,1,1} \) and a known constant. For \( k^{2}_{1} \), \( \gamma_{2,1,1,\ldots,1,1} \), \( \gamma_{3,1,1,\ldots,1,1} \), ... up to \( \gamma_{p_1,1,1,\ldots,1,1} \) - a total of \( p_1-1 \) \( \gamma \)'s - can all be transformed in the same way into the form of known constants multiplied by \( \gamma_{1,1,\ldots,1,1} \). Multiplying \( \gamma_{1,1,\ldots,1,1} \) with different known constants and summing them can replace \( \gamma_{2,1,\ldots,1,1} \) to \( \gamma_{p_1,1,\ldots,1,1} \), these \( p_1-1 \) \( \gamma \)'s. Thus, it can be deduced that \( \gamma_{2,1,1,\ldots,1,1} \) to \( \gamma_{p_1,1,p_3,\ldots, p_m} \), a total of \( p_1p_3\cdots p_m - p_3p_4\cdots p_m \) \( \gamma \)'s, can be represented by \( \gamma_{1,1,\ldots,1,1} \) to \( \gamma_{1,1,p_3,\ldots ,p_m} \), a total of \( p_3p_4 \cdots p_m \) \( \gamma \)'s. By substituting \( \gamma_{2,1,1,\ldots,1,1} \) to \( \gamma_{p_1,1,p_3,\ldots,p_m} \) into the mixed products of \( \gamma_{1,1,\ldots,1,1} \) to \( \gamma_{1,1,p_3\ldots p_m} \) with constants \( \lambda \) and \( k \), the formula can be consolidated into a format that only contains known constants formed by multiplying or dividing \( \lambda \) and \( k \), which are then multiplied and added with \( \gamma_{1,1,\ldots,1,1} \) to \( \gamma_{1,1,p_3,\ldots, p_m} \), a total of \( p_3p_4\cdots p_m \) \( \gamma \)'s.

\[
\overbrace{\overbrace {\gamma_{\underbrace{\scriptstyle1,1,\ldots,1,1}_{m}}\ldots\gamma_{\underbrace{\scriptstyle1,1,p_3,\ldots p_m}_{m}}}^{p_3p_4\cdots p_m}\overbrace {\gamma_{\underbrace{\scriptstyle2,1,\ldots,1,1}_{m}}\ldots\gamma_{\underbrace{\scriptstyle2,1,p_3,\ldots, p_m}_{m}}}^{p_3p_4\cdots p_m}\ldots\overbrace {\gamma_{\underbrace{\scriptstyle p_1,1,\ldots,1,1}_{m}}\ldots\gamma_{\underbrace{\scriptstyle p_1,1,p_3,\ldots, p_m}_{m}}}^{p_3p_4\cdots p_m}}^{p_1p_3p_4\cdots p_m}
\]

For the composition of \( \gamma \) in the numerator of \( k^{2}_{1} \), the arrangement can be divided into \( p_1 \) groups, each of size \(p_3p_4\cdots p_m\). The \( \gamma \)'s in the second to the \(p_2 \)th groups can be transformed according to the equations of \( k^{1}_{1} \) to \( k^{1}_{p_1-1} \). Specifically, these \( \gamma \)'s can be converted into a product of corresponding \( \gamma \)'s from the first group and known constants.

Therefore, by analogy, for the denominator of \( k^{2}_{1} \), the terms from \( \gamma_{1,2,\ldots,1,1} \) to \( \gamma_{p_1,2,p_3,\ldots ,p_m} \) can be transformed into a format where \( p_3p_4 \cdots p_m \) specific \( \gamma \)'s, namely from \( \gamma_{1,2,1,\ldots,1,1} \) to \( \gamma_{1,2,p_3,\ldots, p_m} \), are multiplied by known constants and then summed up. This transformation leverages the known proportional relationships and constants to simplify the original expression.

After integration, the equation for \( k^{2}_{1} \) can be transformed into a format where the numerator is the sum of \(p_3p_4 \cdots p_m  \) specific \( \gamma \)'s, from \( \gamma_{1,1,\ldots,1,1} \) to \( \gamma_{1,1,p_3,\ldots, p_m} \), each multiplied by known constants, and the denominator is similarly the sum of \( p_3p_4 \cdots p_m \) specific \( \gamma \)'s, from \( \gamma_{1,2,1,\ldots,1,1} \) to \( \gamma_{1,2,p_3,\ldots, p_m} \), each multiplied by known constants. This transformation simplifies the original proportional relationship into a more manageable form.

Here, we can organize the coefficient of \( \gamma_{1,1,\ldots,1,1} \) to demonstrate that in the reorganized formula, the known constant formed by multiplying and dividing \( k \) and \( \lambda \), then adding them together, is positive. At this point, the coefficient is given by \( \lambda_{1,1,\ldots,1,1}^{2}+\frac{\lambda_{2,1,\ldots,1,1}^{2}\lambda_{1,1,\ldots,1,1}^{1}}{\lambda_{2,1,\ldots,1,1}^{1} k^{1}_{1}} + \frac{\lambda_{3,1,\ldots,1,1}^{2}\lambda_{1,1,\ldots,1,1}^{1}}{\lambda_{3,1,\ldots,1,1}^{1} k^{1}_{1}} + \ldots + \frac{\lambda_{p_1,1,\ldots,1,1}^{2}\lambda_{1,1,\ldots,1,1}^{1}}{\lambda_{p_1,1,\ldots,1,1}^{1}k^{1}_{1}} \),and it is positive.

Similarly, it can be deduced that the equation for \( k^{2}_{2} \) transforms into a format where the numerator consists of the sum of \( p_3 p_4 \cdots p_m\) specific \( \gamma \)'s, from \( \gamma_{1,1,\ldots,1,1} \) to \( \gamma_{1,1,p_3,\ldots, p_m} \), each multiplied by known constants. The denominator, in parallel, is the sum of \(  p_3p_4 \cdots p_m \) specific \( \gamma \)'s, from \( \gamma_{1,3,1,\ldots,1,1} \) to \( \gamma_{1,3,p_3,\ldots ,p_m} \), each multiplied by known positive constants. This approach simplifies the proportional relationship into a more straightforward and manageable format.

This process can be continued in a similar fashion, leading up to the equation for \( k^{2}_{p_2-1} \). For \( k^{2}_{p_2-1} \), the equation can be transformed into a format where the numerator is the sum of \(  p_3p_4 \cdots p_m \) specific \( \gamma \)'s, ranging from \( \gamma_{1,1,\ldots,1,1} \) to \( \gamma_{1,1,p_3,\ldots p_m} \), each multiplied by known positive constants. Similarly, the denominator is the sum of \( p_3p_4 \cdots p_m \) specific \( \gamma \)'s, ranging from \( \gamma_{1,p_2,1,\ldots,1,1} \) to \( \gamma_{1,p_2,p_3,\ldots, p_m} \), each multiplied by known positive constants. This sequence of transformations simplifies the original complex proportional relationships.

At this stage, for the equations from \( k^{2}_{1} \) to \( k^{2}_{p_2-1} \), both the numerator and denominator have been transformed into a format where \( p_3p_4 \cdots p_m \) specific \( \gamma \)'s are multiplied by known positive constants and then summed up. Furthermore, in the final organized formula, it is established that the coefficients of all \( \gamma \) terms are positive. Now, we repeat the process used previously for \( k^{1}_{1} \) to \( k^{1}_{p_1-1} \). This involves arranging the products of \( \gamma \) and constants in the numerator and denominator in ascending order of the \( \gamma \) subscripts, ensuring each term corresponds one-to-one. Then, extract the first product from both the numerator and denominator, and set the result of the numerator divided by the denominator as the current proportion \( k \). Repeat this operation for the second to the \( p_3p_4 \cdots p_m \)th products.

Taking the example of representing \( \gamma_{1,2,1,\ldots,1,1} \) in terms of \( \gamma_{1,1,\ldots,1,1} \) in \( k^{2}_{1} \), from previous reasoning, we know that the product term containing \( \gamma_{1,1,\ldots,1,1} \) in the numerator can be converted into a known positive constant multiplied by \( \gamma_{1,1,\ldots,1,1} \), and the product term containing \( \gamma_{1,2,1,\ldots,1,1} \) in the denominator can be converted into another known positive constant multiplied by \( \gamma_{1,2,1,\ldots,1,1} \). Therefore, \( \gamma_{1,1,\ldots,1,1} \) multiplied by a known positive constant divided by \( \gamma_{1,2,1,\ldots,1,1} \) multiplied by another known positive constant equals \( k^{2}_{1} \). This implies that \( \gamma_{1,2,1,\ldots,1,1} \) can be represented as \( \gamma_{1,1,\ldots,1,1} \) multiplied by a new known positive constant, which is a combination of the original known positive constants and \( k \).

From the case of \( k^{2}_{1} \), it can be deduced that \( \gamma_{1,1,\ldots,1,1} \) multiplied by a known positive constant can represent \( \gamma_{1,2,1,\ldots,1,1} \), and similarly, \( \gamma_{1,1,\ldots,1,2} \) multiplied by a known positive constant can represent \( \gamma_{1,2,1,\ldots,1,2} \), and this reasoning can be extended up to \( \gamma_{1,1,p_3,\ldots, p_m} \) multiplied by a known positive constant representing \( \gamma_{1,2,p_3,\ldots, p_m} \). Similarly, in the cases of \( k^{2}_{2} \) to \( k^{2}_{p_2-1} \), the same substitution can be made where \( \gamma_{1,1,\ldots,1,1} \) to \( \gamma_{1,1,p_3,p_4,\ldots, p_m} \) replace \( \gamma_{1,3,1,\ldots,1,1} \) to \( \gamma_{1,3,p_3,p_4,\ldots,p_m } \) all the way up to \( \gamma_{1,p_2,1,\ldots,1,1} \) to \( \gamma_{1,p_2,p_3,\ldots, p_m} \).

At this point, we can represent all other \( \gamma \)'s using a total of \( p_3p_4 \cdots p_m \) numbers, ranging from \( \gamma_{1,1,\ldots,1,1} \) to \( \gamma_{\scriptstyle1,1,p_3,p_4,\ldots, p_m} \), by following the above strategy. In the cases of \( k^{3}_{1} \) to \( k^{3}_{p_3-1} \), all other \( \gamma \)'s can be represented using \( p_4p_5\cdots p_6 \) numbers, from \( \gamma_{1,1,\ldots,1,1} \) to \( \gamma_{1,1,1,p_4,p_5,\ldots p_m} \). Similarly, this process can be extrapolated to the cases of \( k^{m}_{1} \) to \( k^{m}_{p_m-1} \), where all \( \gamma \)'s can be represented by a single number, \( \gamma_{1,1,\ldots,1,1} \). For any given \( \delta^{th} \) group's \( k \) (where \( 1 \leq \delta \leq m \), and \( \delta \) is a positive integer), in \( k^{\delta}_{1} \) to \( k^{\delta}_{m-1} \), all other \( \gamma \) terms can be represented by a sequence from \( \gamma_{1,1,\ldots,1,1} \) to \( \gamma_{\underbrace{\scriptstyle1,1,\ldots,1,1}_{\delta},p_{\delta+1},p_{\delta + 2},\ldots, p_{m-1},p_m} \).Furthermore, in the final organized formula, it is established that the coefficients of all \( \gamma \) terms are positive.

Based on the previously discussed reasoning and recursive processing relationships, we can generalize the rule for the arrangement of \( \lambda \) and \( \gamma \) in the numerator and denominator of any proportion \( k \) and its recursive processing. For any given \( k^{\delta}_{\zeta} \) in the \( \delta^{th} \) group and \( \zeta^{th} \) proportion (where \( 1 \leq \delta \leq m \), \( 1 \leq \zeta \leq p_\delta-1 \), and \( \delta \), \( \zeta \) are positive integers), the following equation relationship exists:

\scalebox{0.58}{%

$\frac{\overbrace{\overbrace{\lambda^{\delta}_{\underbrace{\scriptstyle1,1,\ldots,1}_{m}} \gamma_{\underbrace{\scriptstyle1,1,\ldots,1}_{m}} +\ldots+ \lambda^{\delta}_{\underbrace{\scriptstyle1,1,\ldots,1,1}_{\delta},p_{\delta+1},p_{\delta + 2},\ldots, p_{m-1},p_m} \gamma_{\underbrace{\scriptstyle1,1,\ldots,1,1}_{\delta},p_{\delta+1},p_{\delta + 2},\ldots, p_{m-1},p_m} }^{p_{\delta+1}p_{\delta +2}\cdots p_m} + \overbrace{ \ldots+\lambda^{\delta}_{\underbrace{\scriptstyle p_1,\ldots, p_{\delta-1},1,}_{\delta}\underbrace{\scriptstyle p_{\delta+1},\ldots, p_m}_{m-\delta}} \gamma_{\underbrace{\scriptstyle p_1,\ldots, p_{\delta-1},1,}_{\delta}\underbrace{\scriptstyle p_{\delta+1},\ldots, p_m}_{m-\delta}}}^{p_1p_2\cdots p_{\delta-1}p_{\delta+1}\cdots p_m-p_{\delta+1}p_{\delta +2}\cdots p_m}}^{p_1p_2\cdots p_{\delta-1}p_{\delta+1}\cdots p_m}}{\overbrace{\overbrace{\lambda^{\delta}_{\underbrace{\scriptstyle1,1,\ldots, 1,(\zeta+1)}_{\delta}\underbrace{\scriptstyle1,\ldots, 1}_{m-\delta}} \gamma_{\underbrace{\scriptstyle1,1,\ldots, 1,(\zeta+1)}_{\delta}\underbrace{\scriptstyle1,\ldots, 1}_{m-\delta}} +\ldots +\lambda^{\delta}_{\underbrace{\scriptstyle1,1,\ldots, 1,(\zeta+1)}_{\delta}\underbrace{\scriptstyle p_{\delta+1}p_{\delta+2}\ldots p_m}_{m-\delta}} \gamma_{\underbrace{\scriptstyle1,1,\ldots, 1,(\zeta+1)}_{\delta}\underbrace{\scriptstyle p_{\delta+1}p_{\delta+2}\ldots p_m}_{m-\delta}} }^{p_{\delta+1}p_{\delta +2}\cdots p_m}+\overbrace{\ldots +\lambda^{\delta}_{\underbrace{\scriptstyle p_1,\ldots, p_{\delta-1},(\zeta+1),}_{\delta}\underbrace{\scriptstyle p_{\delta+1},\ldots, p_m}_{m-\delta}} \gamma_{\underbrace{\scriptstyle p_1,\ldots, p_{\delta-1},(\zeta+1),}_{\delta}\underbrace{\scriptstyle p_{\delta+1},\ldots, p_m}_{m-\delta}}}^{p_1p_2\cdots p_{\delta-1}p_{\delta+1}\cdots p_m-p_{\delta+1}p_{\delta +2}\cdots p_m}}^{p_1p_2\cdots p_{\delta-1}p_{\delta+1}\cdots p_m}}= k^{\delta}_{\zeta}$
}

Based on the general law of equation relationships, we can summarize any step representing the recursive relationship. According to the reasoning in the previous text, in the numerator and denominator of \(k\) in the \(\delta^{th}\) group, we can represent the remaining \(\gamma\)s by multiplying the known constants formed by \(\lambda\) and \(k\) with a sequence from \(\gamma_{1,1,\ldots,1,1}\) to \(\gamma_{\underbrace{\scriptstyle1,1,\ldots,1,1}_{\delta},p_{\delta+1},p_{\delta+2},\ldots, p_m}\). That is, using the first \(p_{\delta+1}p_{\delta+2}\cdots p_m \gamma\)s in the formula enclosed in brackets to represent the remaining \(p_1p_2\cdots p_{\delta-1}p_{\delta+1}\cdots p_m-p_{\delta+1}p_{\delta+2}\cdots p_m \) \(\gamma\)s. Thus, both the numerator and denominator of the formula can be simplified to \(p_{\delta+1}p_{\delta+2}\cdots p_m \) terms of the product of \(\gamma_{1,1,\ldots,1,1}\) to \(\gamma_{\underbrace{\scriptstyle1,1,\ldots,1,1}_{\delta},p_{\delta+1},p_{\delta+2},\ldots, p_m}\)  with known constants. As deduced previously, the coefficients of \(\gamma\) are all positive. Assuming the rearranged coefficients are \(\iota\), the numerator is arranged from \(\iota_1\) to \(\iota_{p_{\delta+1}p_{\delta+2}\cdots p_m }\), and the denominator from \(\iota_{p_{\delta+1}p_{\delta+2}\cdots p_m +1}\) to \(\iota_{2p_{\delta+1}p_{\delta+2}\cdots p_m }\).

\scalebox{0.72}{%
$\frac{\overbrace{  \iota_{1} \gamma_{\underbrace{\scriptstyle1,1,\ldots,1,1}_{m}} +\iota_{2} \gamma_{\underbrace{\scriptstyle1,1,\ldots,1,2}_{m}} +\ldots+\iota_{p_{\delta+1}p_{\delta+2}\cdots (p_m -1)} \gamma_{\underbrace{\scriptstyle1,1,\ldots,1}_{\delta}\underbrace{\scriptstyle,p_{\delta+1},p_{\delta+2},\ldots, p_m}_{m-\delta}}+\iota_{p_{\delta+1}p_{\delta+2}\cdots p_m } \gamma_{\underbrace{\scriptstyle1,1,\ldots,1}_{\delta}\underbrace{\scriptstyle,p_{\delta+1},p_{\delta+2},\ldots,p_m}_{m-\delta}}}^{p_{\delta+1}p_{\delta+2}\cdots p_m}}{\overbrace{  \iota_{p_{\delta+1}p_{\delta+2}\cdots p_m+1} \gamma_{\underbrace{\scriptstyle1,1,\ldots, 1,(\zeta+1)}_{\delta}\underbrace{\scriptstyle,1,\ldots, 1}_{m-\delta}} +\iota_{p_{\delta+1}p_{\delta+2}\cdots p_m+2} \gamma_{\underbrace{\scriptstyle1,1,\ldots ,1,(\zeta+1)}_{\delta}\underbrace{\scriptstyle,1,\ldots ,1,2}_{m-\delta}} +\ldots+\iota_{2p_{\delta+1}p_{\delta+2}\cdots p_m} \gamma_{\underbrace{\scriptstyle1,1,\ldots, 1,(\zeta+1)}_{\delta}\underbrace{\scriptstyle,p_{\delta+1},p_{\delta+2},\ldots, p_m}_{m-\delta}}}^{p_{\delta+1}p_{\delta+2}\cdots p_m}} = k^{\delta}_{\zeta}
$}

Continuing with the arrangement of product terms in the numerator and denominator according to the ascending order of \(\gamma\) subscripts, each term is made to correspond one-to-one. Once again, set the ratio of corresponding product terms in the numerator and denominator equal to the current \(k\).
\scalebox{1}{%
$
\frac{ \iota_{1} \gamma_{\underbrace{\scriptstyle1,1,\ldots,1,1}_{m}} }{\iota_{p_{\delta+1}p_{\delta+2}\cdots p_m+1} \gamma_{\underbrace{\scriptstyle1,1,\ldots, 1,(\zeta+1)}_{\delta}\underbrace{\scriptstyle,1,\ldots, 1}_{m-\delta}}} = k^{\delta}_{\zeta}\ldots\ldots\frac{ \iota_{p_{\delta+1}p_{\delta+2}\cdots p_m } \gamma_{\underbrace{\scriptstyle1,1,\ldots,1}_{\delta}\underbrace{\scriptstyle,p_{\delta+1},p_{\delta+2},\ldots,p_m}_{m-\delta}} }{\iota_{2p_{\delta+1}p_{\delta+2}\cdots p_m} \gamma_{\underbrace{\scriptstyle1,1,\ldots, 1,(\zeta+1)}_{\delta}\underbrace{\scriptstyle,p_{\delta+1},p_{\delta+2},\ldots, p_m}_{m-\delta}}} = k^{\delta}_{\zeta}
$}

At this stage, the product of \(\gamma_{1,1,\ldots,1,1}\) to \(\gamma_{\underbrace{\scriptstyle1,1,\ldots,1}_{\delta},p_{\delta+1},\ldots, p_m}\)  with known constant terms can represent \(\gamma_{1,1,\ldots,1,(\zeta+1),\underbrace{\scriptstyle1,\ldots,1}_{m-\delta}}\)  to \(\gamma_{1,1,\ldots,1,(\zeta+1),p_{\delta+1},\ldots ,p_m}\) .

\end{proof}

\bibliography{VO2.bib}

\end{document}